\documentclass{conm-p-l}
\usepackage[colorlinks,plainpages,backref=page,urlcolor=blue]{hyperref}
\usepackage{amscd,amssymb,amsmath,amsbsy,amsthm}
\usepackage{txfonts}
\usepackage{mathrsfs}
\usepackage{latexsym}
\usepackage{enumitem}
\usepackage{tikz}
\usetikzlibrary{cd}
\usetikzlibrary{calc}
\usepackage{breakurl}
\usepackage{xurl}
\hypersetup{breaklinks=true,  colorlinks=true,  pdfusetitle=true}

\newtheorem{theorem}{Theorem}[section]
\newtheorem{lemma}[theorem]{Lemma}
\newtheorem{corollary}[theorem]{Corollary}
\newtheorem{proposition}[theorem]{Proposition}

\theoremstyle{definition}
\newtheorem{definition}[theorem]{Definition}
\newtheorem{example}[theorem]{Example}
\newtheorem{remark}[theorem]{Remark}

\newcommand{\Z}{\mathbb{Z}}

\newcommand{\R}{\mathbb{R}}
\newcommand{\C}{\mathbb{C}}

\newcommand{\K}{\mathbb{K}}

\newcommand{\CP}{\mathbb{CP}}
\newcommand{\RP}{\mathbb{RP}}

\renewcommand{\k}{\Bbbk}

\DeclareMathAlphabet{\pazocal}{OMS}{zplm}{m}{n}
\newcommand{\RR}{\pazocal{R}}
\newcommand{\RRR}{\widetilde{\RR}}

\newcommand{\m}{{\mathfrak{m}}}

\DeclareMathOperator{\rank}{rank}

\DeclareMathOperator{\im}{im}

\DeclareMathOperator{\id}{id}
\DeclareMathOperator{\ab}{{ab}}

\DeclareMathOperator{\Sym}{Sym}
\DeclareMathOperator{\ch}{char}

\DeclareMathOperator{\Sq}{{Sq}}
\DeclareMathOperator{\Hom}{{Hom}}

\DeclareMathOperator{\ev}{ev}

\DeclareMathOperator{\Grass}{Gr}
\DeclareMathOperator{\PD}{PD}

\DeclareMathOperator{\D}{d}
\DeclareMathOperator{\MC}{MC}
\DeclareMathOperator{\Conf}{Conf}

\DeclareMathAlphabet{\pazocal}{OMS}{zplm}{m}{n}

\newcommand{\surj}{\twoheadrightarrow}
\newcommand{\inj}{\hookrightarrow}
\newcommand{\isom}{\xrightarrow{
   \,\smash{\raisebox{-0.65ex}{\ensuremath{\scriptstyle\simeq}}}\,}}

\newcommand{\abs}[1]{\left| #1 \right|}

\DeclareMathAlphabet\mathbfcal{OMS}{cmsy}{b}{n}

\def\dot{\mathchar"013A}  
\newcommand{\hdot}{{\raise1pt\hbox to0.35em{\huge $\dot$}\hspace{-0.5pt}}} 
\newcommand{\bwedge}{\mbox{\normalsize $\bigwedge$}}

\newcommand{\cga}{\ensuremath{\textsc{cga}}}
\newcommand{\cdga}{\ensuremath{\textsc{cdga}}}
\newcommand{\pda}{\ensuremath{\textsc{pda}}}
\newcommand{\pdcdga}{\ensuremath{\textsc{pd-cdga}}}

\definecolor{lime}{HTML}{A6CE39}
\DeclareRobustCommand{\orcidicon}{
	\begin{tikzpicture}
	\draw[lime, fill=lime] (0,0) 
	circle [radius=0.16] 
	node[white] {{\fontfamily{qag}\selectfont \tiny ID}};
	\draw[white, fill=white] (-0.0625,0.095) 
	circle [radius=0.007];
	\end{tikzpicture}
	\hspace{-2mm}
}
\foreach \x in {A, ..., Z}{\expandafter\xdef\csname orcid\x\endcsname{\noexpand\href{%
https://orcid.org/\csname orcidauthor\x\endcsname}{\noexpand\orcidicon}}}

\title[Cohomology, Bocksteins, and resonance mod two]%
{Cohomology, Bocksteins, and resonance varieties in characteristic $2$}

\author[Alexander~I.~Suciu]{Alexander~I.~Suciu$^1$\!\!\orcidA{}}
\address{Department of Mathematics,
Northeastern University,
Boston, MA 02115, USA}
\email{\href{mailto:a.suciu@northeastern.edu}{a.suciu@northeastern.edu}}
\urladdr{\href{https://suciu.sites.northeastern.edu}%
{https://suciu.sites.northeastern.edu}}
\thanks{$^1$Supported in part by Simons Foundation Collaboration 
Grants for Mathematicians \#354156 and \#693825}



\begin{document}

\begin{abstract}
We use the action of the Bockstein homomorphism on the cohomology ring 
$H^{\hdot}(X,\Z_2)$ of a finite-type CW-complex $X$ in order to define the 
resonance varieties of $X$ in characteristic $2$. Much of the theory is done 
in the more general framework of the Maurer--Cartan sets and the resonance 
varieties attached to a finite-type commutative differential graded algebra. 
We illustrate these concepts with examples mainly drawn from closed 
manifolds, where Poincar\'e duality over $\Z_2$ has strong implications 
on the nature of the resonance varieties.
\end{abstract}

\maketitle
\setcounter{tocdepth}{1}
\tableofcontents

\section{Introduction}
\label{sect:intro}

\subsection{Resonance varieties}
\label{subsec:res-intro}
The cohomology ring of a space $X$ captures deep, albeit incomplete information 
about the homotopy type of $X$.  A fruitful idea, which arose in the late 1990s from 
the theory of hyperplane arrangements \cite{Fa97,CS99,LY00} turns the cohomology 
algebra $H^{\hdot}(X,\k)$ over a coefficient field $\k$ into a family of cochain complexes  
with differentials given by multiplication by an element in the vector space $H^{1}(X,\k)$. 
One extracts from these data the {\em resonance varieties}\/ of $X$ over $\k$, as the 
loci where the cohomology of the aforementioned cochain complexes jumps in a 
certain degree, by a specified amount.  

Originally, the ground field was $\C$ and the space $X$ was the complement 
of an arrangement of $n$ hyperplanes in $\C^{\ell}$, in which case 
each of these resonance varieties was shown to be a finite union of linear 
subspaces in $\C^n$. Starting with \cite{MS00}, the scope of investigation 
of the resonance varieties broadened, with fields $\k$ of arbitrary characteristic 
and more general spaces $X$ being allowed, as long as 
$X$ had the homotopy type of a connected, finite-type CW-complex, 
its homology groups $H_{\hdot}(X,\Z)$ were free abelian, and the cohomology 
ring $H^{\hdot}(X,\Z)$ was generated in degree $1$. New phenomena were 
discovered in this wider generality; for instance, the resonance varieties over 
$\C$ no longer need to be linear, while the resonance varieties over $\Z_p$ 
(for $p$ a prime) carry useful information regarding the index-$p$ subgroups 
of the second nilpotent quotient of $\pi_1(X)$ (see also \cite{Su-conm}). 
Further investigations of the resonance varieties over finite fields were 
done in \cite{Fa07}. 

The theory was developed in even greater 
generality in \cite{PS-tams, PS-plms10}, where the only condition required was 
that either $\ch(\k)\ne 2$, or $\ch(\k)=2$ and $H_1(X,\Z)$ to be torsion-free, 
so as to insure that $a^2=0$ for every $a\in H^1(X,\k)$. 
Unfortunately, this torsion-freeness condition rules out 
many interesting spaces.  In this note, we remedy this situation, 
by redefining the resonance varieties in characteristic $2$, so as to take 
into account the action of the Bockstein operator, $\beta_2=\Sq^1$,  
on the cohomology algebra $H^{\hdot}(X,\Z_2)$. 

\subsection{Commutative differential graded algebras}
\label{subsec:cdga-intro}

The key idea is to appeal to a construction done in an even wider framework.  
This construction, which was first introduced and studied in \cite{MPPS,PS-springer,Su16} 
and is being further developed in \cite{DS-models, Su-sullivan, Su-cdga}, starts with 
a commutative differential graded algebra $(A,\D)$. For every element $a$ in the 
{\em Maurer--Cartan set}\/ of this $\cdga$,
\begin{equation}
\label{eq:mcs}
\MC(A)=\{a\in A^1 \mid  a^2 + \D(a)=0 \in A^2 \} ,
\end{equation}
one constructs a cochain complex $(A^{\hdot} , \delta_{a})$ 
with differentials $\delta^i_a\colon A^i\to A^{i+1}$ the $\k$-linear 
maps given by $\delta^i_{a} (u)= a \cdot u + \D(u)$ for $u \in A^i$.  
The resonance varieties of $A$ (in degree $q\ge 0$ and depth 
$s\ge 0$) are the loci 
\begin{equation}
\label{eq:rvs-intro}
\RR^q_s(A)=\big\{a \in \MC(A) \mid 
\dim_{\k} H^q(A,\delta_a) \ge  s\big\}. 
\end{equation}
Under the assumption that all the graded pieces $A^i$ 
are finite-dimensional,  these sets are Zariski closed subsets of the 
algebraic variety $\MC(A)$; when $A$ has zero differential, 
the resonance varieties are homogeneous, but in general that 
is not the case.

Especially interesting is the case when $(A,\D)$ is a Poincar\'{e} duality 
$\cdga$ of formal dimension $m$, that is, the underlying graded algebra $A^{\hdot}$ 
satisfies Poincar\'{e} duality and the differential $\D$ vanishes on $A^{m-1}$. 
Building on work from \cite{PS-imrn, Su-edinb}, we show in Theorem \ref{thm:mpd} 
that, for such $\cdga$s, the involution $a\mapsto -a$ on $A^1$ restricts to 
isomorphisms 
\begin{equation}
\label{eq:pdres-intro}
\begin{tikzcd}[column sep=18pt]
\RR^q_s(A) \ar[r, "\simeq"]& \RR^{m-q}_s(A)
\end{tikzcd}
\end{equation}
for all $q, s\ge 0$.

\subsection{The Bockstein operator and resonance in characteristic $2$}
\label{subsec:bock-intro}

We apply the general theory outlined above to the case when 
$A=H^{\hdot}(X,\Z_2)$ is the cohomology algebra of a 
connected, finite-type CW-complex $X$, equipped with 
the differential $\beta_2\colon A^{\hdot} \to A^{\hdot + 1}$ 
given by the Bockstein homomorphism associated to the coefficient 
exact sequence $0\to \Z_2\to \Z_4\to \Z_2\to 0$. We denote the resonance 
varieties of the $\cdga$ $(A,\beta_2)$ by $\RRR^q_s(X,\Z_2)$. Slightly 
more generally, for a field $\k$ of characteristic $2$, we define the 
{\em Bockstein resonance varieties}\/ of $X$ as 
\begin{equation}
\label{eq:brv-intro}
\RRR^q_s(X,\k)=\RRR^q_s(X,\Z_2) \times_{\Z_2} \k .
\end{equation}

If the group $H_1(X,\Z)$ has no $2$-torsion, one may also define 
the ``usual" resonance varieties, $\RR^q_s(X,\k)$, as the resonance 
varieties of the $\cdga$ $(A,0)$. It turns out 
that $\RR^1_s(X,\k)=\RRR^1_s(X,\k)$ for all $s\ge 0$, due to 
the vanishing of the Bockstein $\beta_2\colon H^1(X,\Z_2)\to H^2(X,\Z_2)$ 
in this situation. Nevertheless, the resonance varieties $\RR^q_s(X,\k)$ 
and $\RRR^q_s(X,\k)$ may differ in degrees $q>1$; in fact, we give 
examples showing that either variety may be properly included in the 
other when $s=1$. For instance, 
if $M$ is a smooth, closed, non-orientable manifold of dimension $m$ 
and $H_1(M,\Z)$ has no $2$-torsion, we then show in Corollary \ref{cor:r-rr} 
that $\RR^m_1(M,\Z_2)= \{0\}$, whereas $\RRR^m_1(M,\Z_2)= \Z_2$.

\subsection{Poincar\'e duality, orientability, and resonance}
\label{subsec:pd-intro}
In fact, much more can be said in the case when our space is a topological 
manifold $M$. We always assume $M$ to be a closed manifold (that is, 
compact, connected, with no boundary).  Our first main result in this context 
(proved in Corollary \ref{cor:pd-mfd} and Proposition \ref{prop:res2-nonor-mfd}) 
characterizes the orientability of $M$ in terms of its Bockstein $\cdga$ and its 
top-degree Bockstein resonance variety.

\begin{theorem}
\label{thm:intro-1}
Let $M$ be a closed, smooth manifold of dimension $m$. The following 
conditions are equivalent.
\begin{enumerate}[itemsep=2pt]
\item \label{o1}
$M$ is orientable.
\item \label{o2}
$(H^{\hdot}(M,\Z_2),\beta_2)$ is a Poincar\'{e} duality 
differential graded algebra (of formal dimension $m$).
\item \label{o3}
$\RRR^m_1(M,\Z_2) = \{0\}$.
\end{enumerate}
\end{theorem}

Our second main result (proved in Propositions \ref{prop:res0-mfd} 
and \ref{prop:res2-mfd}) is a Poincar\'e duality  analog 
for the resonance varieties (of both types) of closed, orientable manifolds.

\begin{theorem}
\label{thm:intro-2}
Let $M$ be a closed, orientable manifold of dimension $m$.
\begin{enumerate}[itemsep=2pt]
\item \label{m1}
If $\ch(\k)\ne 2$, then $\RR^q_s(M;\k)=\RR^{m-q}_s(M;\k)$ 
for all $q,s\ge 0$.
\item \label{m2}
If $\ch(\k)=2$, then $\RRR^q_s(M;\k)=\RRR^{m-q}_s(M;\k)$ 
for all $q,s\ge 0$.
\end{enumerate}
\end{theorem}

We illustrate these results with several classes of manifolds, including the 
orientable and non-orientable surfaces, some lens spaces and other 
$3$-manifolds, the real projective spaces $\RP^n$, and the Dold manifolds $P(m,n)$.

\subsection{Betti numbers in finite covers}
\label{subsec:covers-intro}
In the final section we discuss some of the ways in which the resonance varieties 
$\RR^q_s(X,\k)$ and $\RRR^q_s(X,\k)$, as well as the Betti numbers 
$b_q(X,\k)=\dim_{\k} H^q(X,\k)$ behave when passing to finite covers. 

For instance, suppose $Y\to X$ is a connected $2$-fold cover classified 
by a non-zero class $\alpha\in H^1(X,\Z_2)$ with $\alpha^2=0$. In 
Proposition \ref{prop:betti-cover} we strengthen results from \cite{Yo20} 
and show that 
\begin{equation}
\label{eq:betti-intro}
b_q(Y,\Z_2)= b_q(X,\Z_2) + 
\dim_{\Z_2} H^q(H^{\hdot}(X,\Z_2),\delta_{\alpha}),
\end{equation}
for all $q\ge 1$. In particular, the mod-$2$ Betti numbers of $Y$ are at least as 
large as those of $X$, which is not necessarily the case when $\alpha^2\ne 0$.

\section{Maurer--Cartan elements and Koszul complexes}
\label{sect:mc-koszul}

We start with the Maurer--Cartan sets and the Koszul complexes 
associated to a commutative differential graded algebra.  

\subsection{Commutative graded algebras}
\label{subsec:cga}

Throughout this work, $\k$ will denote a ground field. 
A  {\em graded $\k$-vector space}\/ is a vector space $A$ over $\k$, 
together with a direct sum decomposition, into vector subspaces, 
$A=\bigoplus_{i\ge 0}A^{i}$, called its graded pieces. An element 
$a\in A^{i}$ is said to be homogeneous; we write $\abs{a}=i$ for its degree. 

A {\em graded algebra}\/ over $\k$ is a graded $\k$-vector space, 
$A^{\hdot}=\bigoplus_{i\ge 0}A^{i}$, equipped with an associative 
multiplication map, $\cdot \colon A\times A\to A$, making $A$ 
into a $\k$-algebra with unit $1\in A^0$ such that 
$\abs{a\cdot b}=\abs{a}+\abs{b}$ for all homogenous 
elements $a,b\in A$.  A graded algebra $A$ is said to be 
{\em graded-commutative}\/ (for short, a $\cga$), if 
$a\cdot b = (-1)^{\abs{a}\abs{b}} b \cdot a$ for all 
homogeneous $a,b\in A$, and {\em strictly commutative}\/ 
if $ab = ba$ for all $a,b\in A$. Note that the two versions 
of commutativity agree when $\ch(\k)=2$.

A morphism between two 
graded algebras is a $\k$-linear map $\varphi\colon A\to B$ such 
that $\varphi$ sends $A^i$ to $B^i$ for all $i\ge 0$ and satisfies 
$\varphi(a\cdot b)=\varphi(a)\cdot \varphi(b)$ for all $a,b\in A$.

We say that a graded algebra $A$ is {\em connected}\/ if $A^0$ 
is the $\k$-span of the unit $1$ (and thus $A^0=\k$). 
We also say that $A$ is of {\em finite-type}\/ (or, locally finite) 
if all the graded pieces $A^{i}$ are finite-dimensional.

Given a vector space $V$, one may construct several basic 
examples of graded algebras from it: the tensor algebra $T=T(V)$, 
the symmetric algebra $S=\Sym(V)$, and the exterior algebra $E=\bwedge(V)$. 
All three algebras are connected; moreover, $E$ is graded-commutative, 
$S$ is strictly commutative, and $T$ has neither of these properties.  
If $V$ is finite-dimensional, then $T$ and $S$ are of 
finite-type and $E$ is finite-dimensional.
 
\subsection{Commutative differential graded algebras}
\label{subsec:cdga}
We now enrich the notion of a $\cga$ by introducing a differential, 
which we require to be compatible with both the grading and the 
multiplication. 
By definition, a {\em commutative differential graded algebra}\/ 
over $\k$ (for short, a $\k$-$\cdga$) is a pair $A=(A^{\hdot},\D)$, 
where  $A^{\hdot}$ is a $\k$-$\cga$ and $\D\colon A\to A$ is a 
{\em graded derivation}; that is, $\D$ is a $\k$-linear map such that 
$\D\circ \D=0$, $\D( A^i)\subseteq A^{i+1}$ for all $i\ge 0$, 
and $\D(a\cdot b) = \D(a)\cdot b +(-1)^{\abs{a}} a \cdot \D(b)$, 
for all homogeneous elements $a,b\in A$. 
 
Using only the underlying cochain complex structure of the $\cdga$, 
we let $Z^i(A)=\ker (\D\colon A^i\to A^{i+1})$ be the subspace of cocycles and 
$B^i(A)=\im (\D\colon A^{i-1}\to A^{i})$ the subspace of coboundaries, 
and define the $i$-th cohomology group of $(A^{\hdot},\D)$ as the quotient 
$\k$-vector space $H^i(A)=Z^i(A)/B^i(A)$.  The direct sum of these 
vector spaces, $H^{\hdot}(A)=\bigoplus_{i\ge 0} H^i(A)$, 
inherits from $A$ the structure of a graded, graded-commutative 
$\k$-algebra. If $A$ is of finite-type, then clearly $H^{\hdot}(A)$ 
is also of finite-type; in this case, we define the Betti numbers 
of $A$ to be the ranks of the graded pieces of $H^{\hdot}(A)$, 
and write them as $b_i(A)\coloneqq \dim_{\k} H^i(A)$.

For a cocycle $z\in Z^i(A)$, we will denote by $[z]\in H^i(A)$ 
the cohomology class it represents. 
Observe that $\D(1)=\D(1\cdot 1)=\D(1) +\D(1)$, and so $\D(1)=0$. 
Therefore, if $A$ is connected, 
the differential $\D\colon A^0\to A^1$ vanishes, and  
so we may identify $H^0(A)=\k$ and $H^1(A)=Z^1(A)$. 

A morphism between two $\cdga$s, $\varphi\colon (A,\D_A)\to (B,\D_B)$, 
is both a map of graded algebras and a cochain map; that is, $\varphi$ is 
a $\k$-linear map that preserves gradings, multiplicative structures, 
and differentials. Denoting by $\varphi^{i}\colon A^{i}\to B^{i}$ 
the restriction of $\varphi$ to $i$-th graded pieces, we have that 
$\varphi^{i}(ab)=\varphi^{i}(a)\varphi^{i}(b)$ and 
$\D_{B}(\varphi^{i}(a)) = \varphi^{i+1}(\D_{A}(a))$, for all $a,b\in A^{i}$. 
The morphism $\varphi$ induces a morphism 
$\varphi^*\colon H^{\hdot} (A)\to H^{\hdot} (B)$ 
between the respective cohomology algebras.  
We say that $\varphi$ is a {\em quasi-isomorphism}\/ if $\varphi^*$ is an 
isomorphism. 

A {\em weak equivalence}\/ between two $\cdga$s, 
$A$ and $B$, is a finite sequence of quasi-isomor\-phisms going 
either way and connecting $A$ to $B$; one such zig-zag of 
quasi-isomorphisms is given in the diagram below,
\begin{equation}
\label{eq:zig-zag}
\begin{tikzcd}
A  & A_1 \ar[l, pos=0.4, "\varphi_1"'] \ar[r, "\varphi_2"] & \cdots 
& A_{\ell-1}   \ar[l] \ar[r, "\varphi_{\ell}"] & B .
\end{tikzcd}
\end{equation}
Note that a weak equivalence induces a well-defined isomorphism 
$H^{\hdot}(A)\cong H^{\hdot}(B)$. If a  weak equivalence between $A$ and $B$ 
exists, we say that the two $\cdga$s are {\em weakly equivalent}, 
and write $A\simeq B$. Clearly, $\simeq$ is an equivalence 
relation among $\cdga$s.   

\subsection{Maurer--Cartan sets}
\label{subsec:mcs}
Let $(A,\D)$ be a $\cdga$ over a field $\k$. 
We define the {\em Maurer--Cartan set}\/ of $A$ by 
\begin{equation}
\label{eq:qa}
\MC(A)=\left\{a\in A^1 \mid  a^2 + \D(a)=0 \in A^2 \right\}.
\end{equation}

If $A^1$ is finite-dimensional, then the set of Maurer--Cartan elements 
is a quadratic algebraic subvariety of the affine space $A^1$. 
Clearly, this variety contains the point $0\in A^1$.
Two general classes of examples are worth singling out. 

\begin{example}
\label{ex:daa2}
Suppose $a^2 + \D(a)=0$ for all $a\in A^1$.  In this case (which 
we will concentrate on in the latter sections), we clearly have $\MC(A)=A^1$. 
\end{example}

\begin{example}
\label{ex:a2}
Suppose $a^2=0$, for every $a\in A^1$---%
a condition that is always satisfied if $\ch(\k)\ne 2$. In this 
case, $\MC(A)=Z^1(A)$.  If, additionally, $A$ is connected, then 
a previous observation gives $\MC(A)=H^1(A)$.  
\end{example}

In all the above examples, the Maurer--Cartan set is a linear subspace of $A^1$.
In general, though, $\MC(A)$ is cut out by quadratic and linear equations; 
here is a simple example illustrating this fact.

\begin{example}
\label{ex:mca-quad}
Let $A=(\Z_2[a_1,a_2],\D)$, with differential given by $\D(a_1)=a_1a_2$ 
and $\D(a_2)=a_2^2$. Then $\MC(A)=\{(x_1,x_2)\in \Z_2^2\mid x_1(x_1+x_2)=0\}$.
\end{example}

The next lemma (which shall be of use in \S\ref{subsec:res-pd}) shows that 
the Maurer--Cartan set in invariant under negation in the ambient vector space.

\begin{lemma}
\label{lem:involution}
Let $(A,\D)$ be a $\k$-$\cdga$. Then the linear involution $A^1\to A^1$, 
$a\mapsto -a$ restricts to an involution $\MC(A)\to \MC(A)$. 
\end{lemma}

\begin{proof}
If $\ch(\k)=2$, then the claim is obviously true. 
On the other hand, if $\ch(\k)\ne 2$, then $\MC(A)=Z^1(A)$, which is 
a linear subspace of $A^1$, and we are done.
\end{proof}

The construction of the Maurer--Cartan sets is functorial. More precisely, 
we have the following lemma, which follows directly from the definitions.

\begin{lemma}
\label{lem:mc-func} 
Let $\varphi\colon A\to B$ be a morphism of $\cdga$s. 
Then the linear map $\varphi^1\colon A^1\to B^1$ 
restricts to a map, $\bar\varphi\colon \MC(A)\to \MC(B)$, 
between the respective subsets; moreover,
$\overline{\psi\circ\varphi}=\bar\psi\circ\bar\varphi$. 
\end{lemma}

If both $A^1$ and $B^1$ are finite-dimensional, then the map 
$\bar\varphi\colon \MC(A)\to \MC(B)$ from above is a morphism of algebraic 
varieties. In particular, if $\varphi\colon A\to B$ is an isomorphism, 
then $\bar\varphi$ is also an isomorphism.

It also follows from the definitions that the construction is 
compatible with restriction and extension of the base field.

\begin{lemma}
\label{lem:mc-basechange} 
Let $A$ be a $\k$-$\cdga$ with $\dim_{\k} A^1<\infty$, and let 
$\k\subset \K$ be a field extension.  Then,
\begin{enumerate}[itemsep=1.5pt]
\item \label{eq:mc-ext}
$\MC(A) = \MC(A\otimes_{\k} \K) \cap A^1$.

\item \label{eq:mc-base}
$\MC(A\otimes_{\k} \K) = \MC(A) \times_{\k} \K$ .
\end{enumerate}
\end{lemma}

\subsection{The cochain complex associated to a Maurer--Cartan element}
\label{subsec:cc}
Let  $(A,\D)$ be a finite-type $\k$-$\cdga$. 
For each element $a\in \MC(A)$, we have 
a cochain complex of finite-dimensional $\k$-vector spaces, 
\begin{equation}
\label{eq:aomoto}
\begin{tikzcd}
(A^{\hdot} , \delta^A_{a})\colon  \ 
A^0 \ar[r, "\delta^0_{a}"] & A^1
\ar[r, "\delta^1_{a}"]
& A^2   \ar[r, "\delta^2_{a}"]& \cdots ,
\end{tikzcd}
\end{equation}
with differentials $\delta^i_a\colon A^i\to A^{i+1}$ the $\k$-linear maps given by 
\begin{equation}
\label{eq:diff}
\delta^i_{a} (u)= a \cdot u + \D(u), 
\end{equation}
for $u \in A^i$.  The fact that these maps are differentials 
is readily verified:  
\begin{align}
 \notag
\delta_a^{i+1}\delta^i_{a} (u)&= a^2 u +a \cdot\D(u)
+ \D(a)\cdot u -a \cdot\D(u) + \D(\D (u)) \\
&= (a^2  + \D(a))\cdot u \\ \notag
&=0,
\end{align}
where at the last step we used the assumption that $a$ 
belongs to $\MC(A)$. We let 
\begin{equation}
\label{eq:betti-twist}
b_i(A,a)\coloneqq \dim_{\k} H^i\big(A,\delta^A_a\big)
\end{equation}
be the Betti numbers of this cochain complex. Observe that $\delta_0=\D$, 
and thus $b_i(A,0)=b_i(A)$.

The cochain complex associated to a Maurer--Cartan element 
enjoys the following naturality property. Related statements (in 
different levels of generality) can be found in \cite{DS-models, Su-edinb}. 

\begin{lemma}
\label{lem:func-cc}
Let $\varphi\colon (A,\D_{A})\to (B,\D_{B})$ be 
a morphism of finite-type $\k$-$\cdga$s. For each $a\in \MC(A)$, the  
map $\varphi$ induces a chain map, 
\begin{equation}
\label{eq:chain-map}
\begin{tikzcd}[column sep=24pt, row sep=22pt]
&\left(A^{\hdot} , \delta^A_{a}\right): \ar[d, "\varphi_a"]
&[-24pt] A^0 \ar[r, "\delta^0_{a}"] \ar[d, "\varphi^0"]
& A^1  \ar[d, "\varphi^1"] \ar[r] 
& \cdots \ar[r] 
& A^i
\ar[r, "\delta^i_{a}"]  \ar[d, "\varphi^i"]
& A^{i+1}   \ar[r, "\delta^{i+1}_{a}"]  \ar[d, "\varphi^{i+1}"]
& \cdots \phantom{,}
\\ 
&\left(B^{\hdot} , \delta^B_{\varphi(a)}\right): 
&[-24pt] B^0 \ar[r, "\delta^0_{\varphi(a)}"] 
& B^1  \ar[r] 
&\cdots  \ar[r] 
& B^i
\ar[r, "\delta^i_{\varphi(a)}"]
& B^{i+1}   \ar[r, "\delta^{i+1}_{\varphi(a)}"]
& \cdots .
\end{tikzcd}
\end{equation}
\end{lemma}

\begin{proof}
Let $u\in A^{i}$. By definition, $\varphi_a(u)=\varphi^i(u)$. Hence,
\begin{align}
\label{eq:delta-phi}
\notag
\varphi_a (\delta^A_a(u) )&= \varphi^{i+1}(a u + \D_A(u)) \\
&=\varphi^1(a)  \varphi^{i}(u) + \D_B(\varphi^{i}(u)) \\  \notag
&=\delta^B_{\varphi(a)} (\varphi_a(u)),
\end{align}
and the claim is proved.
\end{proof}

Consequently, for each $a\in \MC(A)$, the chain map 
$\varphi_a\colon \big(A^{\hdot} , \delta^A_{a}\big) \to 
\big(B^{\hdot} , \delta^B_{\varphi(a)}\big)$ induces 
homomorphisms in cohomology, 
\begin{equation}
\label{eq:chain-hom-coho}
\begin{tikzcd}[column sep=20pt]
\varphi_a^i\colon H^i\left(A,\delta^A_a\right) \ar[r]&
H^i\left(B,\delta^B_{\bar\varphi(a)}\right) .
\end{tikzcd}
\end{equation}

\begin{lemma}
\label{lem:inj-surj}
Let $\varphi\colon (A,\D_{A})\to (B,\D_{B})$ be 
a morphism of finite-type $\k$-$\cdga$s, 
let $a\in \MC(A)$, and let $i$ be a positive integer. 
\begin{enumerate}
\item \label{ph1}
Suppose $\varphi^i$ is injective and $\varphi^{i-1}$ is surjective. 
Then $\varphi^i_a$  is injective.
\item \label{ph2}
Suppose $\varphi^i$ is surjective and $\varphi^{i+1}$ is injective. 
Then $\varphi^i_a$ is surjective.
\end{enumerate}
\end{lemma}

\begin{proof}
\eqref{ph1}
Suppose $\varphi^i_a ([u])=0$, for some $u\in A^i$ with 
$au+\D_A(u)=0$. Then 
$\varphi^i(u)=\varphi^1(a)v+\D_B(v)$, for some $v\in B^{i-1}$. By 
our surjectivity assumption on $\varphi^{i-1}$, there is an 
element $w\in A^{i-1}$ such that $\varphi^{i-1}(w)=v$, 
and so $\varphi^i(u)=\varphi^i(aw)+\varphi^{i}(\D_A(w))$. Our injectivity 
assumption on $\varphi^{i}$ now implies that  $u=aw+\D_A(w)$, and 
so $[u]=0$.

\eqref{ph2}
Let $[v]\in H^i\big(B,\delta^B_{\bar\varphi(a)}\big)$, for some $v\in B^i$ 
with $\varphi^1(a)v+\D_B(v)=0$. Since $\varphi^{i}$ is surjective, there is 
an element $u\in A^i$ such that $v=\varphi^{i}(u)$. We then have 
\begin{align}
\label{eq:phi-diff}
\notag
\varphi^{i+1}(au+\D_A(u))&=\varphi^1(a)\varphi^{i}(u)+\D_B(\varphi^{i}(u))\\
&=\varphi^1(a)v+\D_B(v)\\
&=0, \notag
\end{align}
and, since $\varphi^{i+1}$ is injective, we have that $au+\D_A(u)=0$.
Therefore, $[v]=\varphi^i_a([u])$, and the proof is complete.
\end{proof}

\subsection{A generalized Koszul complex}
\label{subsec:univ aomoto}

We now fix a basis $\{ e_1,\dots, e_n \}$ for the finite-dimensional 
$\k$-vector space $A^1$, and we let $\{ x_1,\dots, x_n \}$ be the 
Kronecker dual basis for the dual vector space $A_1=(A^1)^{\vee}$.  
In what follows, we shall identify the symmetric algebra $\Sym(A_1)$ 
with the polynomial ring $R=\k[x_1,\dots, x_n]$.

The coordinate ring of the affine variety $\MC(A)\subset A^1$ 
is the quotient, $S=R/I$, of the ring $R$ by the defining ideal of $\MC(A)$. 
Consider the cochain complex of free $S$-modules, 
\begin{equation}
\label{eq:univ aomoto}
\begin{tikzcd}[column sep=22pt]
(A^{\hdot} \otimes_{\k} S,\delta_A) \colon 
\cdots \ar[r] 
&A^{i}\otimes_{\k} S \ar[r, "\delta^{i}"]
&A^{i+1} \otimes_{\k} S \ar[r, "\delta^{i+1}"]
&A^{i+2} \otimes_{\k} S \ar[r] 
& \cdots,
\end{tikzcd}
\end{equation}
where the differentials $\delta^i=\delta^i_A$ are the $S$-linear maps defined by 
\begin{equation}
\label{eq:diff-bis}
\delta^{i}(u \otimes s)= \sum_{j=1}^{n} e_j u \otimes s x_j + \D(u) \otimes s
\end{equation}
for all $u\in A^{i}$ and $s\in S$. The fact that this is a cochain complex is 
easily verified; indeed, $\delta^{i+1}\delta^i (u\otimes s)$ is equal to
\begin{align}
\sum_{k} e_k\bigg( \sum_{j} &
e_j u \otimes s x_j +\D{u} \otimes s\bigg) \otimes x_k + \D\bigg(\sum_{j} 
e_j u \otimes s x_j + \D{u} \otimes s\bigg) \notag \\ 
&= \sum_{j,k} e_k e_j  u \otimes s x_j x_k +\sum_k e_k \D{u}\otimes s x_k
-\sum_j e_j \D{u}\otimes s x_j \\
&= 0,\notag
\end{align}
where at the last step we used the fact that $e_k e_j=-e_je_k$.  

\begin{remark}
\label{rem:koszul}
The cochain complex \eqref{eq:univ aomoto} is independent of 
the choice of basis  for the vector space $A^1$.  
Indeed, under the canonical identification $A^1\otimes_{\k}A_1\cong 
\Hom (A^1,A^1)$, the element  $\iota=\sum_{j=1}^{n} e_j  \otimes  x_j$ 
used in defining the differentials $\delta^{i}$ corresponds to the 
identity map of $A^1$.
\end{remark}

\begin{example}
\label{ex:koszul}
Let $E=\bigwedge (e_1,\dots ,e_n)$ be an exterior algebra 
over a field $\k$ of characteristic not equal to $2$ (with 
generators $e_i$ in degree $1$ and with zero 
differential), and let $S=\k[x_1,\dots ,x_n]$ be its Koszul dual.  
Then the cochain complex $(E^{\hdot}\otimes_{\k} S,\delta)$ is 
the classical Koszul complex $K_{\hdot}(x_1,\dots,x_n)$. 
\end{example}

More generally, if the $\cdga$ $A$ has zero differential, 
each boundary map 
$\delta^i \colon A^i\otimes_{\k} S\to  A^{i+1}\otimes_{\k} S$ 
is given by a matrix whose entries are linear forms in the 
variables $x_1,\dots ,x_n$.  If the differential of $A$ is 
non-zero, though, the entries of $\delta^i$ may also have 
non-zero constant terms. We give a simple example, 
extracted from \cite{MPPS, Su16}.

\begin{example}
\label{ex:nonhomog-complex}
Let $\bwedge(a,b)$ be the exterior algebra on generators $a,b$ in 
degree $1$ over a field $\k$ with $\ch(\k)\ne 2$, and let 
$A=(\bwedge(a,b),\D)$ be the $\cdga$ with differential 
given by $\D(a)=0$ and $\D(b)=b\cdot a$.  Then 
$\MC(A)=H^1(A)$ is $1$-dimensional, generated by $a$. 
Writing $S=\k[x]$, the cochain complex \eqref{eq:univ aomoto} 
takes the form 
\begin{equation}
\label{eq:toy-complex}
\begin{tikzcd}[ampersand replacement=\&, column  sep=44pt]
S \ar[r, "{\delta^0 = {\left(\begin{smallmatrix} x& 0\end{smallmatrix}\right)}}"]
\& S^2  \ar[r, "\delta^1 = \left(\begin{smallmatrix} 0 \\ x-1\end{smallmatrix}\right)"] \& S.
\end{tikzcd}
\end{equation}
\end{example}

The relationship between the cochain complexes \eqref{eq:aomoto} 
and \eqref{eq:univ aomoto} is given by the following  lemma, which 
is a slight generalization of a known result (see e.g.~\cite{Su16}).

\begin{lemma}
\label{lem:two aom}
The specialization of the cochain complex $(A\otimes_{\k} S,\delta)$ 
at an element $a\in \MC(A)$ coincides with the cochain 
complex $(A,\delta_{a})$. 
\end{lemma}

\begin{proof}
Write $a=\sum_{j=1}^{n} a_j e_j\in A^1$, and let 
$\m_a=(x_1-a_1,\dots , x_n-a_n)$ be the maximal ideal at $a$.  
The evaluation map $\ev_a \colon S\to S/\m_a=\k$ is the ring 
morphism given by $g\mapsto g(a_1,\dots, a_n)$.  The resulting 
cochain complex, $A (a)= A\otimes_S S/\m_a$, 
has differentials $\delta^i(a)\colon A^i\to A^{i+1}$ given by
\begin{align}
\label{eq:deltai}
\notag
\delta^i(a)(u) & =\sum_{j=1}^{n} e_j u \otimes \ev_a(x_j) + \D(u)\\
&= \sum_{j=1}^{n} e_j u\cdot  a_j +\D(u)\\
&= a\cdot u + \D(u), \notag
\end{align}
for all $u\in A^i$. Thus, $A (a)=(A,\delta_{a})$, as claimed.
\end{proof}

\section{Resonance varieties of commutative differential graded algebras}
\label{sect:dga}

In this section we study the resonance varieties associated to a 
finite-type commutative differential graded algebra over an arbitrary field.  

\subsection{Resonance varieties}
\label{subsec:res}
Let $(A,\D)$ be a finite-type $\cdga$ over a field $\k$. 
Computing the homology of the cochain complexes $(A,\delta_{a})$ 
for various values of the parameter $a\in \MC(A)$ 
and keeping track of the resulting Betti numbers carves     
out noteworthy subsets of the Maurer--Cartan set of $A$.  

\begin{definition}
\label{def:resvars}
The {\em resonance varieties}\/ (in degree $q\ge 0$ and depth $s\ge 0$) 
of a finite-type $\k$-$\cdga$ $(A,d)$ are the sets 
\begin{equation}
\label{eq:rvs}
\RR^q_s(A)=\big\{a \in \MC(A) \mid 
\dim_{\k} H^q(A,\delta_a) \ge  s\big\}. 
\end{equation}
\end{definition}

As we shall see below, these sets are,  
in fact, Zariski closed subsets of $\MC(A)$. Note that, 
for a fixed $q\ge 0$, the degree-$q$ resonance varieties  
form a descending filtration of the Maurer--Cartan set,  
\begin{equation}
\label{eq:resfilt}
\MC(A)=\RR^q_0(A)\supseteq \RR^q_1(A)\supseteq  \RR^q_2(A) 
\supseteq \cdots .
\end{equation}

Here is a more concrete description of these varieties, 
which follows at once from the definitions.

\begin{lemma}
\label{lem:res-a}
Fix integers $q\ge 1$ and $s\ge 0$.
An element $a\in \MC(A)$ belongs to $\RR^q_s(A)$ if 
and only if there exist elements $u_1,\dots ,u_s\in A^q$ such that 
$au_1+\D (u_1)=\cdots =au_s+\D (u_s)=0$ 
in $A^{q+1}$, and the set $\{av + \D (v),u_1,\dots ,u_s\}$ 
is linearly independent in $A^q$, for all $v\in A^{q-1}$.
\end{lemma}

When the algebra $A$ is connected, the differential $\D\colon A^0 \to A^1$ 
vanishes, and the next lemma readily follows.

\begin{lemma}
\label{lem:res-a-conn}
Let $(A,\D)$ be a connected, finite-type $\k$-$\cdga$. Then,
\begin{enumerate}
\item \label{res1}
The point $0\in \MC(A)$ belongs to  $\RR^q_s(A)$ 
if and only if $s\le b_q(A)$. 
\item \label{res2}
$\RR^0_1(A)=\{0\}$ and $\RR^0_s(A)=\emptyset$ for $s>1$. 
\item \label{res3}
A non-zero element $a\in \MC(A)$ belongs to $\RR^1_1(A)$ if and only if 
there is an element $u\in A^1$ not proportional to $a$ and 
satisfying $au+\D (u)=0$.
\item \label{res4}
If $\dim_{\k} A^1=1$, then $\RR^1_1(A)=\{0\}$ if $b_1(A)=1$ and 
$\RR^1_1(A)=\emptyset$ if $b_1(A)=0$. 
\end{enumerate}
\end{lemma}

\subsection{Equations for the resonance varieties}
\label{subsec:resvar-eq}
Once again, let $(A,\D)$ be a finite-type $\k$-$\cdga$. 
By definition, an element $a \in \MC(A)$ belongs to $\RR^q_s(A)$ 
if and only if 
\begin{equation}
\label{eq:rank delta}
\rank \delta^{q-1}_a + \rank \delta^{q}_a \le c_q -s ,
\end{equation}
where $c_q=\dim_{\k} A^q$. Let $r=\k[x_1,\dots, x_n]$, let $S=R/I$ be the coordinate ring 
of $\MC(A)$, and let $(A^{\hdot} \otimes_{\k} S,\delta_A)$ be the cochain 
complex of free, finitely generated $S$-modules from \eqref{eq:univ aomoto}.
It follows from Lemma \ref{lem:two aom} that 
\begin{equation}
\label{eq:rika}
\RR^q_s(A)= 
V \Big( I_{c_q-s+1} \big(\delta^{q-1}_A\oplus \delta^{q}_A\big) \Big)\, .
\end{equation}
Here, $\oplus$ denotes block-sum of matrices; for an 
$m\times n$ matrix $\psi$ with entries in $S$, we let 
$I_r(\psi)$ denote the ideal of $r\times r$ minors of $\psi$, 
with the convention that $I_0(\psi)=S$ and $I_r(\psi)=0$ if 
$r>\min(m,n)$; finally, $V(\mathfrak{a})\subset \MC(A)$ denotes 
the zero set of an ideal $\mathfrak{a}\subset S$.
This shows that the resonance sets $\RR^q_s(A)$ are indeed subvarieties 
of the Maurer--Cartan variety $\MC(A)\subset A^1$. 

\begin{example}
\label{ex:toy}
Let $(A,\D)$ be the $\cdga$ with  
$A=\Z_2[a_1,a_2]/( a_1^2,a_2^3, a_1a_2)$, where $\abs{a_i}=1$ 
and with differential given by $\D(a_1)=0$, $\D(a_2)=a_2^2$. 
Then $\MC(A)$ is equal to $A^1=\Z_2^2$, 
and it follows that $S=\Z_2[x_1,x_2]/(x_1^2+x_1,x_2^2+x_2)$, 
see the discussion from \S\ref{subsec:aomoto} below. The chain 
complex $(A\otimes_{\Z_2} S,\delta)$ now takes the form 
\begin{equation*}
\label{eq:s1-p2}
\begin{tikzcd}[ampersand replacement=\&, column  sep=44pt]
S \ar[r, "{\delta^0=\left(\begin{smallmatrix} x_1& x_2\end{smallmatrix}\right)}"]
\& S^2  \ar[r, "\delta^1=\left(\begin{smallmatrix} 0 \\ x_2+1\end{smallmatrix}\right)"] \& S,
\end{tikzcd}
\end{equation*}
By formula \eqref{eq:rika}, the resonance varieties $\RR^1_s(A)$ 
are the zero loci of the ideals of minors of size $3-s$ of the matrix 
\begin{equation*}
\label{eq:mat-block}
\delta^0\oplus \delta^1 = \begin{pmatrix} 
x_1 & x_2 & 0 \\0 & 0& 0\\ 0 & 0 &x_2+1\end{pmatrix}.
\end{equation*}
Thus, $\RR^1_0(A)=\Z^2_2$, $\RR^1_1(A)= \{x_1x_2+x_1=0\}$, 
and $\RR^1_2(A)=\emptyset$.
\end{example}

\subsection{Naturality properties}
\label{subsec:res-natural}

The next result describes a (partial) functoriality property 
enjoyed by the resonance varieties. Other results of this type 
(in various levels of generality, though not in the generality 
we work in here) can be found in \cite{MPPS, Su-edinb}.

\begin{proposition}
\label{prop:res-func}
Let $\varphi\colon (A,\D_A) \to (B,\D_B)$ be a morphism of finite-type 
$\cdga$s. Suppose $\varphi^i$ is an isomorphism for all $i\le q$ and a 
monomorphism for $i= q+1$, for some $q\ge 0$.
\begin{enumerate}
\item \label{i1}
If $q=0$, the map $\bar\varphi\colon \MC(A)\to \MC(B)$ 
is an embedding which sends
$\RR^1_s(A)$ into $\RR^1_s(B)$ for all $s\ge 0$.
\item  \label{i2}
If $q\ge 1$, the map $\bar\varphi\colon \MC(A)\to \MC(B)$ 
is an isomorphism which identifies
$\RR^i_s(A)$ with $\RR^i_s(B)$ for all $i\le q$ 
and sends $\RR^{q+1}_s(A)$ into $\RR^{q+1}_s(B)$,
for all $s\ge 0$.
\end{enumerate}
\end{proposition}

\begin{proof}
\eqref{i1} First assume that $q=0$. 
By our hypothesis, the map $\varphi^1$ is injective, 
and thus $\bar\varphi$ is also injective. 
Let $a\in \MC(A)$. By Lemma \ref{lem:func-cc}, part \eqref{ph1}, the 
chain map  $\varphi_a\colon \big(A^{\hdot} , \delta^A_{a}\big) \to 
\big(B^{\hdot} , \delta^B_{\bar\varphi(a)}\big)$ induces a monomorphism, 
$\varphi_a^1\colon H^1\big(A,\delta^A_a\big) \to 
H^1\big(B,\delta^B_{\bar\varphi(a)}\big)$.  Now suppose $a\in \RR^1_s(A)$, 
that is, $\dim_{\k} H^1\big(A,\delta^A_a\big)\ge s$; then 
$H^1\big(B,\delta^B_{\bar\varphi(a)}\big)\ge s$, 
and so $\bar\varphi(a)\in \RR^1_s(B)$. 
This shows that $\bar\varphi(\RR^1_s(A))\subseteq \RR^1_s(B)$, 
and the first claim is proved.

\eqref{i2}
Next, assume that $q\ge 1$. By hypothesis, the map $\varphi^1$ is an isomorphism;  
thus, $\bar\varphi$ is also an isomorphism. For each $a\in \MC(A)$, 
Lemma \ref{lem:func-cc} implies that the 
homomorphism $\varphi_a^i\colon H^i\big(A,\delta^A_a\big) \to 
H^i\big(B,\delta^B_{\bar\varphi(a)}\big)$ is an isomorphism for 
$i\le q$ and a monomorphism for $i=q+1$. The second claim follows.
\end{proof}

A similar argument yields the following proposition, 
which shows that the resonance varieties of a $\k$-$\cdga$ 
$(A,d)$ only depend on the isomorphism type of $A$.

\begin{proposition}
\label{prop:res-iso}
Let $\varphi\colon A\to B$ be an isomorphism of $\cdga$s.  
Then $\varphi$ restricts to an isomorphism  
$\bar\varphi\colon \MC(A)\to \MC(B)$ which identifies 
$\RR^q_s(A)$ with $\RR^q_s(B)$ for all $q, s\ge 0$.  
\end{proposition}

The conclusions of Proposition \ref{prop:res-iso} do not always hold  
if we only assume that the map $\varphi\colon A\to B$ is a quasi-isomorphism.  
We give a simple example of this phenomenon, adapted from 
\cite{MPPS, Su16}.  

\begin{example}
\label{ex:dbab}
Let $A=(\bwedge(a,b),\D)$ be the $\cdga$ from Example \ref{ex:nonhomog-complex}, 
with $\D(a)=0$ and $\D(b)=ba$, and let $A'=(\bwedge(a), 0)$ be the sub-$\cdga$ 
generated by $a$. It is readily seen that the inclusion map, 
$\varphi\colon A'\inj A$, induces an isomorphism 
$\varphi^*\colon H^{\hdot}(A')\isom H^{\hdot}(A)$; moreover, $\varphi^1$ 
identifies $\MC(A')=\k$ with $\MC(A)$. On the other hand, 
$\RR^1_1(A')=\{0\}$ is strictly included in $\RR^1_1(A)=\{0,1\}$.
\end{example}

Noteworthy is the fact that the resonance varieties are compatible 
with restriction and extension of the base field.

\begin{lemma}
\label{lem:res-basechange} 
Let $A$ be a finite-type $\k$-$\cdga$, and let $\k\subset \K$ be 
a field extension. Then, for all $q,s\ge 0$,
\begin{enumerate}[itemsep=1.5pt]
\item \label{eq:res-ext}
$\RR^q_s(A) = \RR^q_s(A\otimes_{\k} \K) \cap \MC(A)$.
\item \label{eq:res-base}
$\RR^q_s(A\otimes_{\k} \K) = \RR^q_s(A) \times_{\k} \K$.
\end{enumerate}
\end{lemma}

\begin{proof}
By \eqref{eq:rika}, the resonance varieties $\RR^q_s(A)$ are determinantal 
varieties of matrices defined over $\k$. The two claims follow.
\end{proof}
 
\subsection{Resonance varieties of tensor products}
\label{sub:resonance prod}
Let $(A,d_A)$ and $(B,d_B)$ be two $\k$-$\cdga$s.  
The tensor product of these two $\k$-vector spaces, 
$A\otimes_{\k} B$, has a natural $\cdga$ structure, with grading 
$(A\otimes_{\k} B)^{q}=\bigoplus_{i+j=q} A^i \otimes_{\k} B^j$, 
multiplication $(a\otimes b)\cdot (a'\otimes b')=
(-1)^{\abs{b}\abs{a'}} (ab\otimes a'b')$, and differential 
$D$ given on homogeneous elements by $D(a\otimes b)=
d_A(a) \otimes b +(-1)^{\abs{a}} a\otimes d_B(b)$. 

If both $A$ and $B$ are connected, finite-type $\cdga$s then 
clearly $A\otimes_{\k} B$ is also a connected, finite-type $\cdga$. 
Moreover, upon identifying $(A\otimes_{\k} B)^1$ with 
$A^1\oplus B^1=A^1\times B^1$, 
we have
\begin{equation}
\label{eq:mc-tensor}
\MC(A\otimes_{\k} B)=\MC(A)\times \MC(B).
\end{equation}

The resonance varieties of a tensor product of $\cdga$s obey 
a type of K\"{u}nneth formula. Such formulas 
were obtained in \cite{PS-plms10, SW-mz, Su-edinb} under the assumption 
that the differentials vanish; a proof of one inclusion was given in 
full generality in \cite{PS-springer}.  We give here a proof which is 
both complete and in full generality.

\begin{proposition}
\label{prop:prod-res}
Let $(A,d_A)$ and $(B,d_B)$ be two connected, finite-type $\cdga$s. Then,  
\begin{enumerate}[itemsep=2pt]
\item\label{t1}
$\RR^q_1(A \otimes_{\k} B)=\bigcup_{i+j=q} \RR^i_1(A)\times  \RR^j_1(B)$, 
for all $q\ge 1$.
\item \label{t2}
$\RR^1_s(A \otimes_{\k} B)=\RR^1_s(A)\times \{0\} \cup \{0\}\times \RR^1_s(B)$, 
for all $s\ge 1$.
\end{enumerate}
\end{proposition}

\begin{proof}
Set $C=A\otimes_{\k} B$, and 
let $c = (a, b)$ be an element in $\MC(C)=\MC(A)\times \MC(B)$. 
The cochain complex $(C,\delta^C_c)$ splits as a tensor product of cochain 
complexes, $(A,\delta^A_a)\otimes_{\k} (B,\delta^B_b)$. Therefore,
\begin{equation}
\label{eq:bettiAomoto}
b_q(C,c)=\sum_{i+j=q} b_i(A,a) \cdot b_j(B,b),
\end{equation}
and part \eqref{t1} follows. When $q=1$, we obtain from \eqref{eq:bettiAomoto}
the following equalities: 
\begin{gather}
\begin{aligned}
b_1(C,(0,0))&=b_1(A, 0)+b_1(B, 0), 
\\ 
b_1(C,(a,0))&=b_1(A, a) &&\text{if  $a\ne 0$}, 
\\
b_1(C,(0,b))&=b_1(B, b) &&\text{if $b\ne 0$},  
\\
b_1(C,(a,b))&=0 &&\text{if $a\ne 0$ and $b\ne 0$}. 
\end{aligned}
\end{gather}
Part \eqref{t2} readily follows.
\end{proof}

\subsection{Resonance varieties of coproducts}
\label{sub:resonance coprod}
Let $(A,d_A)$ and $(B,d_B)$ be two  
connected $\cdga$s.  Their wedge sum, $A\vee B$, is a 
new connected $\cdga$, whose underlying graded vector space in 
positive degrees is the direct sum $A^+\oplus B^+$, whose multiplication 
is given by $(a,b)\cdot (a',b') = (aa', bb')$, and whose differential is 
$D=d_A+d_B$. Note that $(A\vee B)^1=A^1\oplus B^1$; therefore, we may 
identify the Maurer--Cartan set $\MC(A\vee B)$ with the product 
$\MC(A)\times \MC(B)$.

The following proposition recovers a result proved 
in various levels of generality in \cite{PS-plms10, PS-springer, 
SW-mz, Su-edinb}.

\begin{proposition}
\label{prop:rescoprod}
Let $C=A\vee B$ be the coproduct of two connected, finite-type 
$\cdga$s with $b_1(A)>0$ and $b_1(B)>0$. Then, we have for all $s\ge 1$
\begin{equation}
\label{eq:res-wedge}
\RR^q_s(A\vee B)=
\begin{cases}
\ \bigcup\limits_{j+k=s-1} \RR^1_j(A) \times \RR^1_k(B)
&\quad \text{if $q=1$}, 
\\[8pt]
\hspace{6pt} \bigcup\limits_{j+k=s}\hspace{6pt} 
\RR^q_j(A)\times  \RR^q_k(B)
&\quad \text{if $q\ge 2$}.
\end{cases}
\end{equation}
\end{proposition}

\begin{proof}
Let $c=(a, b)$ be an element in $\MC(C)=\MC(A)\times \MC(B)$.  
The chain complex $(C,c)$ splits (in positive degrees) as a direct sum 
of chain complexes, $(C^{+},\delta^C_c) \cong (A^{+},\delta^A_a) \oplus 
(B^{+},\delta^B_b)$. The Betti numbers of the respective chain complexes 
are related, as follows:
\begin{equation}
\label{eq:beti-cc}
b_q(C,c)=\begin{cases}
b_q(A, a)+b_q(B, b)+1&\text{if $q=1$ 
and $a\ne 0$, $b\ne 0$,}
\\[3pt]
b_q(A, a)+b_q(B, b)&\text{otherwise.}
\end{cases}
\end{equation}
The claim now follows by a case-by-case analysis of 
formula \eqref{eq:beti-cc}.
\end{proof}

\section{Poincar\'{e} duality algebras and resonance varieties}
\label{sect:pd}

In this section we study the interplay between Poincar\'{e} duality, 
the Koszul complex, and the resonance varieties of a differential 
graded algebra.

\subsection{Poincar\'{e} duality}
\label{subsec:def pd}
We start with a basic algebraic concept that abstracts the classical notion 
of Poincar\'{e} duality for manifolds. For more information on this subject, we 
refer to \cite{MeS, SS10, Su-edinb}.

\begin{definition}
\label{def:pda}
Let $A$ be a connected, finite-type $\cga$ over a field $\k$. We say that 
$A$ is a {\em Poincar\'{e} duality $\k$-algebra}\/ of 
formal dimension $m$ (for short, an $m$-$\pda$) if there is a 
$\k$-linear map $\varepsilon\colon A^m \to \k$ (called an 
{\em orientation}) such that the bilinear forms 
\begin{equation}
\label{eq:duality}
A^i \otimes_{\k} A^{m-i}\to \k, \quad a\otimes b\mapsto \varepsilon (ab)
\end{equation}
are non-singular, for all $i\ge 0$.
\end{definition}

It follows  straight from the definition that the map $\varepsilon$ is an isomorphism, 
and that $A^i=0$ for $i>m$.  Furthermore, for each $0\le i \le m$, there is an isomorphism
\begin{equation}
\label{eq:pd}
\PD^i\colon A^{i}\to (A^{m-i})^*, \quad  \PD^i(a)(b)=\varepsilon (ab).
\end{equation}

Consequently, each element $a\in A^i$ has a Poincar\'{e} 
dual, $a^{\vee}\in A^{m-i}$, which is uniquely determined by the 
formula $\varepsilon (a a^{\vee})=1$. 
We define the orientation class $\omega_A\in A^m$ as the 
Poincar\'{e} dual of $1\in A^0$, that is, $\omega_A=1^{\vee}$. 
Conversely, a choice of orientation 
class $\omega_A\in A^m$ defines an orientation 
$\varepsilon\colon A^m \to \k$ by setting $\varepsilon(\omega_A)=1$.

\begin{definition}
\label{def:pdcdga}
A $\k$-$\cdga$ $(A^{\hdot},\D)$ is a {\em Poincar\'{e} duality differential graded 
algebra}\/ of formal dimension $m$ (for short, an $m$-$\pdcdga$) if 
\begin{enumerate}[itemsep=1pt]
\item \label{pdc1}
The underlying graded algebra $A^{\hdot}$ is an $m$-$\pda$.
\item \label{pdc2}
$\D( A^{m-1})=0$. 
\end{enumerate}
\end{definition}

Condition \eqref{pdc2} can also be stated as $\varepsilon(\D(u))=0$ for all 
$u\in A^{m-1}$. By condition \eqref{pdc1}, the algebra $A$ is connected 
and $A^m \cong A^0$; thus, condition \eqref{pdc2} is equivalent to $H^m(A)=\k$. 
As noted in \cite{LS-aif}, if $(A^{\hdot},\D)$ 
is an $m$-$\pdcdga$, then $H^{\hdot}(A)$ is an $m$-$\pda$, 
with orientation $[\varepsilon]\colon H^m(A) \to \k$ given by 
$[\varepsilon] ([u])=\varepsilon (u)$ for every $u\in Z^m(A)$. 
Lambrechts and Stanley also showed in \cite{LS-asens} that if $A$ is 
a $\cdga$ such that $H^1(A)=0$ and $H^{\hdot}(A)$ is an $m$-$\pda$, 
then $A$ is weakly equivalent to an $m$-$\pdcdga$.

Clearly, if $A$ is an $m$-$\pda$, then $A$ endowed with the 
zero differential is an $m$-$\pdcdga$. 
Also, if $(A,\D)$ is a $\cdga$ such that $A^1=0$ and $A$ is an 
$m$-$\pda$, then $A^{m-1}=0$, and so $(A,\D)$ is $m$-$\pdcdga$. 
We give a simple example of a $\pdcdga$ $(A,\D)$ 
with $\D\ne 0$ and $A^1\ne 0$.

\begin{example}
\label{ex:pdcdga}
Let $A=(\Z_2[x]/(x^{2k}),\D)$, where $\abs{x}=1$, $k\ge 2$, and 
$\D(x)=x^2$. Clearly, $A$ is a Poincar\'e duality algebra 
of dimension $m=2k-1$; since $\D(x^{2k-2}) =0$, we have that 
$A$ is an $m$-$\pdcdga$.
\end{example}

Next, we provide an example of a $\cdga$ $(A,\D)$ such that 
$A$ is an $m$-$\pda$, yet  condition \eqref{pdc2} from 
Definition \ref{def:pdcdga} is not satisfied.

\begin{example}
\label{ex:non-pdcdga}
Let $A=(\Z_2[x,y]/(x^2,y^2),\D)$ with $\abs{x}=1$, $\abs{y}=k$, and differential 
given by $\D(x)=0$, $\D(y)=xy$. Clearly, the underlying graded algebra $A^{\hdot}$ is 
a $\Z_2$-Poincar\'{e} duality algebra of dimension $k+1$, with orientation 
class $\omega_{A}=xy$. Nevertheless, $\D(y)\ne 0$, and so 
$(A,\D)$ is not a $\pdcdga$.
\end{example}

\subsection{Resonance varieties of PD-$\cdga$s}
\label{subsec:res-pd}

Poincar\'{e} duality imposes some rather stringent constraints on the resonance varieties 
of a $\pdcdga$ $A$. To ascertain those constraints, we start by analyzing the 
boundary maps in the Aomoto complexes $(A,\delta_a)$. Versions of the next 
lemma were given in \cite[Lemma 7.4]{PS-imrn} when $\ch(\k)=0$ 
and in \cite[Lemma 5.1]{Su-edinb} when $\ch(\k)\ne 2$ and $\D=0$. 
For completeness and to fix a sign issue in the latter reference, we 
give a self-contained proof here, valid in full generality.

\begin{lemma}
\label{lem:cd-uptosign}
Let $(A^{\hdot} , \D)$ be an $m$-$\pdcdga$ over an arbitrary field $\k$.  Then, 
for all $a\in A^1$ and all $0\le i\le m$, we have a commuting square,
\begin{equation}
\label{eq:cd-sign}
\begin{tikzcd}[column sep=42pt, row sep=28pt]
(A^{m-i})^*   \ar[r, "(\delta^{m-i-1}_{-a})^*"]
& (A^{m-i-1})^*  \\
 A^{i}   \ar[r, "\delta^i_{a}"] \ar[u, "\PD^i"', "\cong"]
 &A^{i+1}\, . \ar[u, "(-1)^{i+1}  \PD_{i+1}"', "\cong"]
\end{tikzcd}
\end{equation}
\end{lemma}

\begin{proof}
Let $b \in A^{i}$ and $c \in A^{m-i-1}$.  Then
\begin{gather}
\begin{aligned}
\label{eq:c1}
(-1)^{i+1}  \PD^{i+1}\circ \delta^{i}_{a}(b) (c)  
&=(-1)^{i+1}  \PD^{i+1}( a b + \D(b))  (c) \\
&=(-1)^{i+1} \varepsilon(abc +\D(b)c),
\end{aligned}
\end{gather}
while 
\begin{gather}
\begin{aligned}
\label{eq:c2}
(\delta^{m-i-1}_{-a})^*\circ \PD^{i} (b) (c)
&=\PD^i (b) ( \delta^{m-i-1}_{-a} (c) )  \\
&=\PD^i(b) (-a c + \D(c)) \\
&=\varepsilon(-bac +b\D(c)). 
\end{aligned}
\end{gather}

Now, since $A$ is a $\cdga$, we have that $ab=(-1)^{i}ba$ 
and $\D(b) c+(-1)^{i} b \D(c)=\D(b c)$. 
Moreover, since $A$ is an $m$-$\pdcdga$ and $bc\in A^{m-1}$, 
we have that $\D(bc)=0$, and the claim follows.
\end{proof}

We are now in a position to state and prove the main result of this section, 
which recovers results from \cite{PS-imrn, Su-edinb} and extends them to 
our present context. First recall from Lemma \ref{lem:involution} that 
the Maurer--Cartan set $\MC(A)$ is invariant under the involution of $A^1$ 
sending $a$ to $-a$.

\begin{theorem}
\label{thm:mpd}
Let $(A^{\hdot},\D)$ be a Poincar\'{e} duality $\cdga$ of 
formal dimension $m$ over a field $\k$. Then,
\begin{enumerate}[itemsep=1pt]
\item \label{respd1} 
$H^i(A,\delta_{a})^* \cong H^{m-i}(A,\delta_{-a})$ 
for all $a\in \MC(A)$ and $i\ge 0$.
\item \label{respd2} 
The linear isomorphism $A^1 \isom A^1$, $a\mapsto -a$ restricts 
to isomorphisms $ \RR^i_s(A) \isom  \RR^{m-i}_s(A)$ 
for all $i, s\ge 0$.
\item \label{respd3} 
$\RR^m_1(A)=\{0\}$.
\end{enumerate}
\end{theorem}

\begin{proof}
Part \eqref{respd1} is a direct consequence of Lemma \ref{lem:cd-uptosign}. 
Part \eqref{respd2} follows from \eqref{respd1}. Finally, Part \eqref{respd3} 
follows from \eqref{respd2} and the fact that $A$ is connected, and so 
$\RR^0_1(A)=\{0\}$. 
\end{proof}

\section{The  Aomoto--Bockstein complex}
\label{sect:bock}

We now use the mod-$2$ cohomology algebra 
of a topological space $X$, equipped with the differential given by the 
Bockstein operator to define the Aomoto--Bockstein complex of $X$.

\subsection{The Bockstein operator}
\label{subsec:bock}
We start by reviewing some standard material on Bocksteins, following the 
treatment in Hatcher \cite[\S{3.E}]{Ha}. For each prime $p$, 
let $A=H^{\hdot}(X,\Z_p)$ be the cohomology algebra of $X$ with $\Z_p$ 
coefficients, with multiplication given by the cup-product. For each $q\ge 0$, 
we have a {\em Bockstein operator},  
$\beta_p\colon H^q(X,\Z_p)\to H^{q+1}(X,\Z_p)$, which is defined as the 
coboundary homomorphism associated to the coefficient exact sequence 
\begin{equation}
\label{eq:zp-bock}
\begin{tikzcd}[column sep=20pt]
0 \ar[r] & \Z_{p}
\ar[r, "\times p"] & \Z_{p^2}
\ar[r] & \Z_{p}
\ar[r] & 0.
\end{tikzcd}
\end{equation}

Alternatively, if $\beta_0\colon H^q(X,\Z_p)\to H^{q+1}(X,\Z)$ is the 
coboundary map associated to the coefficient sequence 
$0\to \Z\to \Z\to \Z_p\to 0$ and $\rho_p\colon H^{\hdot}(X,\Z)\to H^{\hdot}(X,\Z_p)$
is the morphism induced by the projection $\Z\surj \Z_p$, 
then $\beta_p=\rho_p\circ \beta_0$.

Note that a cohomology class $u\in H^q(X,\Z_p)$ is the reduction mod-$p$ 
of a cohomology class in $H^q(X,\Z_{p^2})$ if and only if $\beta_p(u)=0$. 
Consequently, if $u=\rho_p(v)$ for some $v\in H^q(X,\Z)$, then $\beta_p(u)=0$. 
Hence, if all the integral cohomology groups of $X$ are torsion-free, the 
Bockstein $\beta_p$ vanishes identically.

The Bockstein operators in all degrees $q\ge 0$ assemble into a 
$\Z_p$-linear map, $\beta_p\colon A\to A$, with the properties that 
$\beta_p\circ \beta_p=0$ and 
$\beta_p(a\cup b) = \beta_p(a)\cup b + (-1)^{\abs{a}}a\cup \beta_p(b)$ for all 
homogeneous elements $a,b\in A$. 
Therefore, $\beta_p$ is a graded derivation of the algebra $A$ and the pair $(A,\beta_p)$ 
is a $\cdga$. Furthermore, the Bockstein is functorial: if $f\colon X\to Y$ 
is a continuous map, and $f^*\colon H^{\hdot}(Y,\Z_p)\to H^{\hdot}(X,\Z_p)$ is the induced 
homomorphism in cohomology, then $\beta^X_p(f^*(u))=f^*(\beta^Y_p(u))$, 
for all $u \in H^{\hdot}(Y,\Z_p)$.

The homology groups of the Bockstein cochain complex $(A,\beta_p)$, 
denoted $BH^{q}(X,\Z_p)$, are called the {\em Bockstein cohomology groups}\/ 
of $X$ (with coefficients in $\Z_p$). If $H^{\hdot}(X,\Z)$ is of finite-type, then 
each $\Z$-summand of $H^{q}(X,\Z)$ contributes a $\Z_p$-summand to 
$BH^{q}(X,\Z_p)$; the $\Z_p$-summands do not contribute to 
Bockstein cohomology; and each $\Z_{p^k}$-summand of $H^{q}(X,\Z)$ with $k>1$ 
contributes a $\Z_p$-summand to both $BH^{q-1}(X,\Z_p)$ and $BH^{q}(X,\Z_p)$. 
In particular, $b_q(X)\le \dim_{\Z_p} BH^{q}(X,\Z_p)$.

When $p=2$, the Bockstein operator $\beta_2$ may be identified with 
the Steenrod square operation $\Sq^1\colon H^{\hdot}(X,\Z_2)\to H^{\hdot}(X,\Z_2)$. 
The fact that $\beta_2=\Sq^1$ is a derivation may be seen as a consequence of 
the \'{A}dem relation $\Sq^1\circ\Sq^1=0$ and the Cartan relation 
$\Sq^1(a\cup b) = \Sq^1(a)\cup \Sq^0(b) +\Sq^0(a)\cup \Sq^1(b)$, 
where $\Sq^0=\id$.

\subsection{The  Aomoto--Bockstein cochain complex}
\label{subsec:aomoto}

Assume now that $X$ is a connected, finite-type CW-com\-plex,  
and consider the cohomology algebra $A=H^{\hdot} (X,\Z_2)$, 
endowed with the differential  given by the 
Bockstein operator $\beta_2=\Sq^1$. 
Since $\Sq^1(a)=a^2$ for all $a\in A^1$, the Maurer--Cartan set 
for the $\cdga$ $(A,\beta_2)$ is then $\MC(A)=A^1$.

\begin{definition}
\label{def:aom bock}
The {\em Aomoto--Bockstein complex}\/ of $A=H^{\hdot} (X,\Z_2)$   
with respect to an element $a\in A^1$ is the cochain 
complex of finite-dimensional $\Z_2$-vector spaces,
\begin{equation}
\label{eq:aomoto-bis}
\begin{tikzcd}[column sep=18pt]
(A , \delta_a)\colon  \ 
A^0\ar[r, "\delta_a"] 
& A^1 \ar[r, "\delta_a"] 
& \cdots \ar[r, "\delta_a"] 
& A^q \ar[r, "\delta_a"] 
& A^{q+1}\ar[r, "\delta_a"] 
& \cdots ,
\end{tikzcd}
\end{equation}
where $\delta_a(u)=au+\beta_2(u)$.
\end{definition}

Note that for $a=0$ this is the usual Bockstein cochain complex, $(A,\beta_2)$; 
therefore, $H^q(A,\delta_0)=BH^q(X,\Z_2)$. On the other hand, if the 
Bockstein operator vanishes identically, then \eqref{eq:aomoto-bis} 
is the usual Aomoto complex, with differentials $\delta_a(u)=au$.

Here is an alternative (and, in some ways, more 
convenient) interpretation. We pick a basis 
$\{ e_1,\dots, e_n \}$ for $A^1=H^1(X,\Z_2)$, we let 
$\{ x_1,\dots, x_n \}$ be the Kronecker dual basis for 
$A_1=H_1(X,\Z_2)$, and we identify the symmetric algebra 
$\Sym(A_1)$ with the polynomial ring $\Z_2[x_1,\dots, x_n]$.   
The coordinate ring of $A^1$ is then the quotient ring 
\begin{equation}
\label{eq:coord ring}
S=\Z_2[x_1,\dots, x_n]/(x_1^2 + x_1,\dots,x_n^2+x_n),
\end{equation}
which may be viewed as the ring of (Boolean) functions on $\Z_2^n$. 
For detailed information on how to compute with polynomials, ideals, 
Gr\"{o}bner bases, and varieties over this ring, we refer to \cite{BD, B-W, Lu}.

\begin{definition}
\label{def:univ aom}
The {\em universal Aomoto--Bockstein complex}\/ of the cohomology algebra 
$A=H^{\hdot} (X,\Z_2)$ is the cochain complex of free $S$-modules, 
\begin{equation}
\label{eq:univ aomoto-bis}
\begin{tikzcd}[column sep=15pt]
(A\otimes_{\Z_2} S,\delta)\colon \:
A^{0}\otimes_{\Z_2} S \ar[r, "\delta^{0}"] \ar[r] 
&A^{1}\otimes_{\Z_2} S \ar[r] & 
\cdots  \ar[r] 
&A^{i}\otimes_{\Z_2} S \ar[r, "\delta^{i}"]
&A^{i+1} \otimes_{\Z_2} S  \ar[r]
& \cdots,
\end{tikzcd}
\end{equation}
where the differentials are defined by 
\begin{equation}
\label{eq:bock diff}
\delta^{i}(u \otimes 1)= \sum_{j=1}^{n} 
e_j u   \otimes x_j + \beta_2(u)  \otimes 1
\end{equation}
for $u\in A^{i}$, and then extended by $S$-linearity.  
\end{definition}

If the Bockstein operator acts trivially on $H^{\hdot}(X,\Z_2)$, then  
each boundary map $\delta^i$ in the cochain complex \eqref{eq:univ aomoto-bis}
is given by a matrix whose entries are linear forms in the 
variables $x_1,\dots ,x_n$. In general, though, the entries of 
$\delta^i$ may also have non-zero constant terms.

\begin{example}
\label{ex:rp-koszul}
Let $X=\RP^{\infty}$ be the infinite-dimensional  real projective space. 
We identify $H^{\hdot}(X,\Z_2)=\Z_2[a]$ where $\abs{a}=1$ 
and note that the Bockstein $\beta_2$ sends $a^{i}$ to $a^{i+1}$ if $i$ is odd and 
to $0$ otherwise. Writing $S=\Z_2[x]/(x^2+x)$, the resulting Aomoto--Bockstein complex 
has boundary maps $\delta^i\colon S\to S$ given by $\delta^i(a^i)=(x+1) a^{i+1}$ 
if $i$ is odd and $\delta^i(a^i)=xa^{i+1}$ if $i$ is even, and thus has the form 
\begin{equation}
\label{eq:ab-rpinfty}
\begin{tikzcd}[column sep=22pt]
S \ar[r, "x"] \ar[r] 
&S \ar[r, "x+1"] 
&S \ar[r, "x"]
&S  \ar[r]
& \cdots.
\end{tikzcd}
\end{equation}
Note that this cochain complex is exact.
\end{example}

\subsection{The twisted Bockstein operator}
\label{subsec:greenblatt}
To conclude this section, we present an alternative interpretation of the 
differentials in the Aomoto--Bockstein complex \eqref{eq:aomoto-bis}, 
due to Samelson \cite{Sam} and Greenblatt \cite{Gr}. 

Fix a basepoint $x_0\in X$ and let $\pi=\pi_1(X, x_0)$ be the fundamental 
group of $X$ at $x_0$. Every mod~$2$ cohomology 
class $w\in H^1(X,\Z_2)$ defines a local system $\Z_{w}$ on $X$, 
as follows. Viewing $w$ as a homomorphism from $H_1(X,\Z)$ to 
$\Z_2$, and identifying the target group with $\{\pm 1\}$, regarded as the units of 
the ring $\Z$ puts a $\Z[H_1(X,\Z)]$-module structure on the additive group $\Z$. 
Restricting scalars to the group ring $\Z[\pi]$ via the abelianization map 
$\pi \surj \pi_{\ab}=H_1(X,\Z)$ yields a $\Z[\pi]$-module structure on $\Z$, 
which is the module that we denote by $\Z_{w}$. 

The twisted Bockstein operator $\beta_{w}\colon H^q(X,\Z_2)\to H^{q+1}(X,\Z_{w})$ 
is defined as the coboundary map induced by the exact sequence 
\begin{equation}
\label{eq:times2}
\begin{tikzcd}[column sep=22pt]
0 \ar[r] & \Z_{w}
\ar[r, "\times 2"] & \Z_{w}
\ar[r] & \Z_2
\ar[r] & 0.
\end{tikzcd}
\end{equation}
Let $\bar\rho_2\colon H^{\hdot}(X,\Z_{w})\to H^{\hdot}(X,\Z_2)$
be the coefficient morphism induced by the projection $\Z_{w}\surj \Z_2$, 
and set $\bar{\beta}_{w} \coloneqq 
\bar\rho_2\circ \beta_{w}\colon H^q(X,\Z_2)\to H^{q+1}(X,\Z_2)$.  
By standard properties of the Bockstein operator, 
we have that $\bar\beta_{w}^2=0$.  Furthermore, as shown in \cite{Gr}, the map 
\begin{equation}
\label{eq:barbeta}
\begin{tikzcd}[column sep=18pt]
\bar{\beta} \colon H^1(X,\Z_2) \times H^q(X,\Z_2) \ar[r]& H^{q+1}(X,\Z_2)
\end{tikzcd}
\end{equation}
given by $\bar\beta(w, u) = \bar{\beta}_{w}(u)$ is a cohomology 
operation.  That is, if $f\colon X\to Y$ is a 
continuous map, and $f^*\colon H^{\hdot}(Y,\Z_2)\to H^{\hdot}(X,\Z_2)$ 
is the induced homomorphism in cohomology, then, for every 
$w\in H^1(X,\Z_2)$ and $u\in H^q(Y,\Z_2)$, we have that
\begin{equation}
\label{eq:natural}
f^*(\bar\beta(w,u))=\bar\beta(f^*(w),f^*(u)).
\end{equation}

\begin{theorem}[\cite{Sam, Gr}]
\label{thm:gblatt}
For every $w \in H^1(X,\Z_2)$ and $u\in H^q(X,\Z_2)$, we have that 
$\bar\beta(w,u)=w u + \Sq^1(u)$. 
\end{theorem}

\section{The resonance varieties of a space}
\label{sect:res vars}

In this section we discuss the resonance varieties of a space, with coefficients 
in a field $\k$. Throughout, $X$ will be a connected, finite-type CW-complex  
and $A=H^{\hdot}(X,\k)$ will be its cohomology algebra over $\k$, with multiplication 
given by the cup-product.

\subsection{The usual resonance varieties}
\label{subsec:res vars}

We start by singling out certain situations in which all cohomology classes 
in degree $1$ square to zero. The next (standard) lemma formalizes arguments 
from \cite{PS-tams, PS-plms10}.

\begin{lemma}
\label{lem:sq0}
Suppose either $\ch(\k)\ne 2$, or $\ch(\k)=2$ and $H_1(X,\Z)$ 
has no $2$-torsion. Then $a^2=0$ for all $a\in H^1(X,\k)$.
\end{lemma}

\begin{proof}
By graded-commutativity of the cup-product, $a^2=-a^2$. Thus, 
if $\ch(\k)\ne 2$, then $a^2=0$. 

Now suppose $\ch(\k)=2$ and $H_1(X,\Z)$ has no $2$-torsion. 
In this case, by the Universal Coefficients Theorem, every class 
$a\in H^1(X,\k)$ is the image of a class $\tilde{a}\in H^1(X,\Z)$ 
under the coefficient homomorphism $\Z\to\k$. By obstruction theory, 
there is a map $f\colon X\to S^1$ and a class $\omega\in H^1(S^1,\Z)$ 
such that $\tilde{a}=f^*(\omega)$. Hence, 
$\tilde{a}\cup\tilde{a}= f^*(\omega\cup\omega) =0$. 
The claim then follows by naturality of cup products with respect to 
coefficient homomorphisms. 
\end{proof}

Under the assumptions of Lemma \ref{lem:sq0}, we shall view the 
cohomology algebra $H^{\hdot}(X,\k)$ as a $\k$-cdga with differential $d=0$ 
and we shall identify the Maurer--Cartan set $\MC(A)$ with the 
affine space $A^1=H^1(X,\k)$. 

\begin{definition}
\label{def:resvars-x} 
Let $X$ be a connected, finite-type CW-complex, let $\k$ be a field, 
and suppose either 
$\ch(\k)\ne 2$, or $\ch(\k)=2$ and $H_1(X,\Z)$ has no $2$-torsion. 
The {\em resonance varieties}\/ of $X$ (over $\k$) are the resonance 
varieties of the $\cdga$ $A=(H^{*}(X,\k),0)$; that is,  
\begin{equation}
\label{eq:res-x}
\RR^q_s(X,\k)=\big\{a \in H^1(X,\k) \mid \dim_{\k} H^q(A,\delta_a) \ge  s\big\}, 
\end{equation}
where $\delta_a\colon A^q \to A^{q+1}$ is given by $\delta_a(u)=au$.
\end{definition}

These sets are homogeneous algebraic subvarieties of the affine space 
$H^1(X,\k)$. For each $q\ge 0$, they form a descending filtration,
\begin{equation}
\label{eq:filt-rr}
H^1(X,\k)=\RR^q_0(X,\k)\supseteq \RR^q_1(X,\k)\supseteq  \RR^q_2(X,\k) 
\supseteq \cdots \supseteq \RR^q_{b_q}(X,\k) \supseteq \{0\},
\end{equation}
where $b_q=b_q(X,\k)\coloneqq \dim_{\k} H^q(X,\k)$; moreover, 
$\RR^q_{b_q+1}(X,\k) =\emptyset$.  By construction, the resonance 
varieties of $X$ depend only on the cohomology algebra $H^{\hdot}(X,\k)$. 
We illustrate the concept with some well-known examples, 
see e.g.~\cite{Su12, Su16}.

\begin{example}
\label{ex:res torus}
Let $T^n$ be the $n$-dimensional torus.  Using the 
exactness of the Koszul complex from Example \ref{ex:koszul}, 
we see that $\RR^i_s(T^n,\k)$ is equal to $\{0\}$ if $0<s\le \binom{n}{i}$  
and is empty for $s>\binom{n}{i}$.
\end{example}

\begin{example}
\label{ex:res surf}
Let $M_g$ be the orientable surface of genus $g>1$. In depth $s>0$ 
the resonance varieties are then given by
\begin{equation}
\label{eq:res surf}
\RR^q_s(M_g,\k)=
\begin{cases}
\k^{2g} & \text{if $q=1$, $s< 2g-1$},\\
\{0\}& \text{if $q=1$, $s\in \{2g-1, 2g\}$ or $q\in\{0, 2\}$, $s=1$},\\
\emptyset & \text{otherwise}.
\end{cases}
\end{equation} 
\end{example}

Best understood are the degree $1$, depth $1$ resonance varieties 
$\RR^1_1(X,\k)$. For low values of $n=b_1(X,\k)$, these varieties are easy to describe. 
If $n=0$, then $\RR^1_1(X,\k)=\emptyset$, while if $n=1$, then $\RR^1_1(X,\k)=\{0\}$.
If $n=2$, then $\RR^1_1(X,\k)=\k^{2}$ or $\{0\}$, according to whether the cup product 
vanishes on $H^1(X,\k)$ or not. 
For $n\ge 3$, the variety $\RR^1_1(X,\k)$ can 
be much more complicated; in particular, it may have 
irreducible components which are not linear subspaces.  

\begin{example}
\label{ex:conf spaces}
Let $X=\Conf(E,3)$ be the configuration space of $3$ labeled points 
on a torus. Then $H^{\hdot}(X,\k)$ is the exterior algebra on generators 
$a_1,a_2,a_3, b_1,b_2,b_3$ in degree $1$, modulo 
the ideal generated by$(a_1-a_2)(b_1-b_2)$, 
$(a_1-a_3)(b_1-b_3)$, and $(a_2-a_3)(b_2-b_3)$. 
A calculation reveals that 
\[
\RR^1_1(X,\k)=\{ (a,b) \in \k^6 \mid 
a_1+a_2+a_3=b_1+b_2+b_3=a_1b_2-a_2b_1=0\}.
\]
Hence, $\RR^1_1(X,\k)$ is isomorphic to $\{a_1b_2-a_2b_1=0\}$, 
a smooth quadric hypersurface in $\k^4$. 
\end{example}

\subsection{Resonance varieties in characteristic $2$}
\label{subsec:res vars 2}
We now treat the case when $\ch(\k)=2$, without imposing any restriction on the 
$2$-torsion of $H_1(X,\Z)$. 

\begin{definition}
\label{def:rvs}
The {\em Bockstein resonance varieties}\/ of a connected, 
finite-type CW-complex $X$ are the  resonance 
varieties of the $\cdga$ $A=(H^{\hdot}(X,\Z_2),\beta_2)$; that is, 
\begin{equation}
\label{eq:rvs-bis}
\RRR^q_s(X,\Z_2)=\big\{a \in H^1(X,\Z_2) \mid 
\dim_{\Z_2} H^q(A,\delta_a) \ge  s\big\},
\end{equation}
where $\beta_2=\Sq^1\colon A^{q}\to A^{q+1}$  
is the Bockstein operator and 
$\delta_a\colon A^q \to A^{q+1}$ is given by 
$\delta_a(u)=au+\beta_2(u)$.

More generally, we define the Bockstein resonance varieties of $X$ 
over a field $\k$ of characteristic $2$ as
\begin{equation}
\label{eq:rvs-char2}
\RRR^q_s(X,\k) = \RRR^q_s(X,\Z_2) \times_{\Z_2} \k\, .
\end{equation}
\end{definition}

These sets are algebraic subvarieties of the affine space $H^1(X,\k)$. For 
each $q\ge 0$, they form a descending filtration,
\begin{equation}
\label{eq:filt-rrr}
H^1(X,\k)=\RRR^q_0(X,\k)\supseteq \RRR^q_1(X,\k)\supseteq  \RRR^q_2(X,\k) 
\supseteq \cdots .
\end{equation}
Since all the essential features of the resonance varieties in 
characteristic $2$ already appear over the field with $2$ elements, 
we will concentrate mainly on the sets $\RRR^q_s(X,\Z_2)$. 

\subsection{Properties of the Bockstein resonance varieties}
\label{subsec:properties-rrr}
As a direct consequence of Lemma \ref{lem:res-a-conn}, we have the 
following result.

\begin{lemma}
\label{lem:res-bock}
Let $X$ be a connected, finite-type CW-complex. Then,
\begin{enumerate}[itemsep=1pt]
\item \label{br1}
The element $0\in H^1(X,\Z_2)$ belongs to $\RRR^q_{s}(X,\Z_2)$ if 
and only if $\dim_{\Z_2} BH^q(X,\Z_2) \ge s$. 
\item \label{br2}
$\RRR^0_1(X,\Z_2)=\{0\}$ and $\RRR^0_s(X,\Z_2)=\emptyset$ 
for $s>1$. 
\end{enumerate}
\end{lemma}

Since $\Sq^1(u)=u^2$ for all $u\in H^1(X,\Z_2)$, the varieties $\RRR^1_s(X,\Z_2)$ 
depend only on the truncated cohomology ring $H^{\le 2}(X,\Z_2)$. In depth $s=1$,  
we can be even more concrete. 

\begin{lemma}
\label{lem:depth1}
Let $X$ be a connected, finite-type CW-complex. Then,
\begin{enumerate}[itemsep=1pt]
\item \label{dep1}
$0\in \RRR^1_1(X,\Z_2)$ 
if and only if there is an element $u\in H^1(X,\Z_2)$ 
such that $u^2=0$. 
\item \label{dep2}
A non-zero element $a\in H^1(X,\Z_2)$ belongs to $\RRR^1_1(X,\Z_2)$ 
if and only if there is an element $u\in H^1(X,\Z_2)$ such that $au+u^2=0$ and 
$u$ is not proportional to $a$. 
\end{enumerate}
\end{lemma}

For instance, if $H^1(X,\Z_2)=\{0,a\}$, then $\RRR^1_1(X,\Z_2)$ is either 
equal to $\{0\}$ or is empty, according to whether $a^2=0$ or not. 

It follows from Proposition \ref{prop:res-func}, part \eqref{i1} that 
both types of degree $1$ resonance varieties, $\RR^1_s(X,\k)$  
and $\RRR^1_s(X,\k)$, enjoy a (partial) naturality property.

\begin{proposition}
\label{prop:r1-nat}
Let $f\colon X\to Y$ be a continuous map, let $\k$ be a field, and suppose that the 
induced homomorphism $f^*\colon H^1(Y, \k) \to H^1(X,\k)$ is injective. 
\begin{enumerate}[itemsep=1pt]
\item \label{map1}
If either $\ch(\k)\ne 2$, or $\ch(\k)=2$ and both $H_1(X,\Z)$ and $H_1(Y,\Z)$
have no $2$-torsion, then $f^*$ restricts to embeddings 
$\RR^1_s(Y,\k) \inj \RR^1_s(X,\k)$ for all $s\ge 1$.
\item \label{map2}
If $\ch(\k)=2$, then $f^*$ restricts to embeddings 
$\RRR^1_s(Y,\k) \inj \RRR^1_s(X,\k)$ for all $s\ge 1$.
\end{enumerate}
\end{proposition}

\subsection{Sample computations of resonance}
\label{subsec:examples-rrr}

We now give several examples illustrating the computation of the 
Bockstein resonance varieties $\RRR^q_s(X,\Z_2)$ for some 
familiar spaces $X$. For simplicity, we will assume throughout 
that $s>0$. 

\begin{example}
\label{ex:klein}
Let $N_g$ be the non-orientable surface of genus $g\ge 1$.  Then 
$H^{\hdot}(N_g,\Z_2)=\Z_2[a_1,\dots,a_g]/I$, where $\abs{a_i}=1$ 
and $I$ is the ideal generated by $a_i^2+a_j^2$ and 
$a_ia_j$ for all $1\le i<j\le g$.  Direct computation shows that 
\begin{equation}
\label{eq:res ng}
\RRR^q_s(N_g,\Z_2) =\begin{cases}
\{0,a_1+\cdots +a_g\} &\text{if $q=1$ and $s=1$},\\
\{0\}  &\text{if $q=0$ or $2$ and $s=1$},\\
\emptyset  &\text{otherwise}.
\end{cases}
\end{equation}
\end{example}

\begin{example}
\label{ex:rpn}
Let $X=\RP^n$ be the real projective space of dimension $n$.  
Then $H^{\hdot}(\RP^n,\Z_2)=\Z_2[a]/(a^{n+1})$, with $a$ in degree $1$, 
and $\beta_2(a^q)=a^{q+1}$ if $q$ is odd, and $0$ otherwise. Thus, 
$BH^q(\RP^n,\Z_2)=0$ for $0<q<n$; moreover, 
$\delta_a(a^q)=a^{q+1}$ if $a$ is even, and $0$ otherwise. 
It follows that 
\begin{equation}
\label{eq:res rpn}
\RRR^q_s(\RP^n,\Z_2) =\begin{cases}
\Z_2 &\text{if $q=n$, $n$ is even, and $s=1$},\\
\{0\}  &\text{if $q=0$ and $s=1$},\\
\emptyset  &\text{otherwise}.
\end{cases}
\end{equation}
\end{example}

\begin{example}
\label{ex:bo}
Let $\Grass_n=\Grass_n(\R^{\infty})$ be the Grassmannian of (unoriented) $n$-planes 
in $\R^{\infty}$. Then $H^{\hdot}(\Grass_n,\Z_2) = \Z_2[w_1, \dots,w_n]$, 
where  $w_k$ is the (universal) Stiefel--Whitney class of degree $k$,  
and the Bockstein operator $\beta_2=\Sq^1$ is given by Wu's formula, 
\begin{equation}
\label{eq:beta-wu}
\beta_2(w_k)=w_1w_k + (k+1) w_{k+1},
\end{equation} 
see \cite{Milnor}. A  straightforward computation (or the observation 
at the end of Example \ref{ex:rp-koszul}) shows that 
$\RRR^q_s(\Grass_1,\Z_2)=\emptyset$ for all $q, s>0$. 
A slightly more involved computation yields
\begin{equation}
\label{eq:res bo-2}
\RRR^{2q}_s(\Grass_2,\Z_2)=\RRR^{2q+1}_s(\Grass_2,\Z_2)=
\begin{cases}
\Z_2  &\text{if $0< s\le q$},\\
\{0\}  &\text{if $q<s\le 2q$},\\
\emptyset  &\text{if $s>2q$}.
\end{cases}
\end{equation}
The resonance varieties $\RRR^{q}_s(\Grass_n,\Z_2)$ for $n>2$ can be computed 
in like fashion.
\end{example}

\subsection{Comparing the two types of resonance}
\label{subsec:res-compare}

We end this section with a comparison between the two types of mod-$2$ 
resonance varieties, $\RR^q_s(X,\Z_2)$ and $\RRR^q_s(X,\Z_2)$. 
We start by observing that these varieties do agree in degree $q=1$, provided 
the former are defined. 

\begin{lemma}
\label{lem:no-2-torsion}
Suppose $H_1(X,\Z)$ has no $2$-torsion. Then 
$\RR^1_s(X,\Z_2)=\RRR^1_s(X,\Z_2)$ for all $s\ge 0$.
\end{lemma}

\begin{proof}
By Lemma~\ref{lem:sq0}, we have that $u^2=0$ for all $u\in H^1(X,\Z_2)$.  
Thus, the Bockstein $\beta_2\colon H^1(X,\Z_2) \to H^2(X,\Z_2)$ vanishes, 
and the claim follows straight from the definitions of the respective resonance 
varieties.
\end{proof}

In degrees $q>1$, though, it may happen that 
$\RR^q_s(X,\Z_2)\ne \RRR^q_s(X,\Z_2)$.  We shall see examples of 
spaces $X$ for which $\RR^q_1(X,\Z_2)\subsetneqq \RRR^q_1(X,\Z_2)$ 
in some degrees $q>1$ in Example \ref{ex:sphere-bundle} below, and also in 
Corollary \ref{cor:r-rr} and in Example \ref{ex:non-orient} in the next section. 
We give now an example of a space $X$ for which 
$\RRR^q_1(X,\Z_2)\subsetneqq \RR^q_1(X,\Z_2)$ for some $q>1$.

\begin{example}
\label{ex:suspension}
Let $X=S^1\vee \Sigma(\RP^2)$ be the wedge of a circle with a 
suspended projective plane. Then $H^{i}(X,\Z_2)=\Z_2$, generated 
by classes $a_i$ for $1\le i\le 3$, with all cup products among these generators 
vanishing; moreover, since $\beta_2=\Sq^1$ commutes with suspensions, 
$\beta_2(a_2)=a_3$. It follows that $\RR^2_1(X,\Z_2)=\Z_2$, whereas 
$\RRR^2_1(X,\Z_2)=\emptyset$.
\end{example}

If the mod-$2$ cohomology ring of $X$ is generated in degree $1$, then 
the Bockstein $\beta_2=\Sq^1$ is completely determined by 
the cup-product structure of $H^{\hdot}(X,\Z_2)$, and so the mod-$2$ 
resonance varieties $\RRR^q_s(X,\Z_2)$ depend only 
on the cohomology ring $H^{\hdot}(X,\Z_2)$. 
In general, though, the varieties $\RRR^q_s(X,\Z_2)$ also depend 
on the action of the Steenrod algebra on the cohomology ring, as 
the next example shows.

\begin{example}
\label{ex:sphere-bundle}
Let $X=S^2\times S^1$ and let $Y=S^2\widetilde{\times} S^1$ be 
the non-trivial $S^2$-bundle over $S^1$, with monodromy given 
by the antipodal map. Clearly, $H_1(X,\Z)=H_2(Y,\Z)=\Z$. 
The mod-$2$ cohomology rings of both spaces are isomorphic to 
$\Z_2[a,b]/(a^2,b^2)$, where $\abs{a}=1$ and $\abs{b}=2$, 
while the two Bocksteins are given by $\beta^X_2(a)=\beta^X_2(b)=0$ and 
$\beta^Y_2(a)=0$, $\beta^Y_2(b)=ab$. The usual resonance varieties $\RR^q_s$ 
of the two spaces are the same (over any field $\k$); in particular, 
$\RR^2_1=\RR^3_1=\{0\}$. Nevertheless, the mod-$2$ 
resonance varieties $\RRR^q_s$ differ:
\begin{align*}
\RRR^1_1(X,\Z_2)&=\{0\},  &\RRR^2_1(X,\Z_2)&=\{0\}, & \RRR^3_1(X,\Z_2)&=\{0\},
\\
\RRR^1_1(Y,\Z_2)&=\{0\}, &\RRR^2_1(Y,\Z_2)&=\Z_2, &\RRR^3_1(Y,\Z_2)&=\Z_2.
\end{align*}
\end{example}

\section{Manifolds, Poincar\'e duality, and resonance}
\label{sect:res-mfd}

In this section we study the resonance varieties of topological manifolds. 
All manifolds will be assumed to be without boundary, compact, and 
connected---for short, {\em closed}\/ manifolds. Most of the theory works 
as well for Poincar\'e complexes, though at some point we will require 
manifolds to be smooth. We start by discussing the interplay between 
Poincar\'e duality, orientability, and the Bockstein operator.

\subsection{Poincar\'e duality and orientability}
\label{subsec:pd-mfd}
Let $M$ be a closed manifold of dimension $m$.
If $M$ is orientable, then its cohomology algebra with coefficients 
in an arbitrary field $\k$ is a Poincar\'{e} duality algebra, also of 
dimension $m$.  We consider here the case when $M$ is not 
necessarily orientable, and restrict our attention to the coefficient 
field $\k=\Z_2$, in which case $H^{\hdot}(M,\Z_2)$ is again a Poincar\'{e} 
duality algebra  of dimension $m$.

We start with a well-known lemma, relating the orientability 
of $M$ to the Bockstein operator $\beta_2=\Sq^1$ on $H^{\hdot}(M,\Z_2)$, 
see Massey \cite[Lemma 1]{Mas}. 
For completeness, we give a full proof, since only a brief 
sketch is given in that reference.

\begin{lemma}[\cite{Mas}]
\label{lem:bock-or}
Let $M$ be a closed manifold of dimension $m$.  Then 
$M$ is orientable if and only if the Bockstein 
$\beta_2\colon H^{m-1}(M,\Z_2)\to H^{m}(M,\Z_2)$ is zero.
\end{lemma}

\begin{proof}
If $M$ is orientable, then $H_{m-1}(M,\Z)\cong H^1(M,\Z)$ is free abelian. Thus, 
the Bockstein operator  $\beta_0\colon H_{m}(M,\Z_2)\to H_{m-1}(M,\Z)$ 
associated to the coefficient sequence $0\to \Z\xrightarrow{\times 2} \Z\to \Z_2\to 0$ 
is trivial. Since  $\beta_2$ is $\Z_2$-dual to the homomorphism 
$\rho_2\circ \beta_0\colon H_{m}(M,\Z_2)\to H_{m-1}(M,\Z_2)$, 
we infer that $\beta_2= 0$. 

If $M$ is non-orientable, then the torsion subgroup of $H_{m-1}(M,\Z)$ 
is equal to $\Z_2$ and may be identified with the image of the Bockstein 
$\beta_0\colon H_{m}(M,\Z_2)\to H_{m-1}(M,\Z)$, 
see \cite[p.~238]{Ha}. Since  $\beta_2= 
(\rho_2\circ \beta_0)^*$, we conclude that $\beta_2\ne 0$. 
\end{proof}

\begin{remark}
\label{rem:wu}
For a smooth $m$-manifold $M$, an alternative proof can be given, using 
the Wu formulas (see also \cite[Corollary 9.8.5]{Hs}). These formulas  
relate the Wu classes $v_i \in H^{i}(M,\Z_2)$, 
the Stiefel--Whitney classes $w_i \in H^{i}(M,\Z_2)$, and the 
Steenrod squares $\Sq^i\colon H^{q}(M,\Z_2)\to H^{q+i}(M,\Z_2)$ via 
\begin{equation}
\label{eq:wu}
w=\Sq(v) \:\text{ and }\: v\cup u=\Sq(u)
\end{equation}
for all $u\in H^{\hdot}(M,\Z_2)$, where $w=\sum_{i\ge 0} w_i$ is the 
total Stiefel--Whitney class and likewise for $v$ and $\Sq$, see \cite{Milnor}. 
Interpreting these formulas in degree $1$, we see that $w_1=v_1$ and 
$v_1$ is Poincar\'e dual to $\varepsilon \circ \Sq^1 \in H^{m-1}(M,\Z_2)^*$.  
Since $\beta_2=\Sq^1$ and since $M$ is orientable if and only if $w_1=0$, 
the desired conclusion follows.
\end{remark}

We now consider the graded $\Z_2$-algebra $A=H^{\hdot}(M,\Z_2)$, viewed as 
a $\cdga$ with differential given by the Bockstein $\beta_2\colon A\to A$.

\begin{corollary}
\label{cor:pd-mfd}
Let $M$ be a closed manifold of dimension $m$. Then
$(H^{\hdot}(M,\Z_2),\beta_2)$ is a Poincar\'{e} duality 
differential graded algebra (of dimension $m$) if and only if $M$ is 
orientable.
\end{corollary}

\begin{proof}
By Poincar\'e duality, the cohomology algebra $H^{\hdot}(M,\Z_2)$ is an  
$m$-$\pda$. The claim follows from Definition \ref{def:pdcdga} 
and Lemma \ref{lem:bock-or}.
\end{proof}

\begin{remark}
\label{rem:Postnikov}
As shown by Postnikov in \cite{Po}, every $3$-dimensional Poincar\'e 
duality algebra over $\Z_2$ can be realized as the cohomology ring, 
$H^{\hdot}(M,\Z_2)$, of some closed (not necessarily orientable) 
$3$-manifold $M$. 
\end{remark}

\subsection{Resonance varieties of closed manifolds}
\label{subsec:res-mfd}

Poincar\'e duality has strong implications on the nature of the resonance 
varieties of closed manifolds (and Poincar\'e complexes). First, as 
an immediate consequence of Theorem \ref{thm:mpd}, we have the 
following result.

\begin{proposition}
\label{prop:res0-mfd}
Let $M$ be a closed, orientable manifold of dimension $m$, 
and let $\k$ be a field of characteristic different from $2$. 
Then $\RR^q_s(M;\k)=\RR^{m-q}_s(M;\k)$ 
for all $q,s\ge 0$. In particular, $\RR^m_1(M,\k)= \{0\}$.
\end{proposition}

Turning to the resonance varieties over a field $\k$ of characteristic $2$, 
we have several results. For the first one, we impose no orientability condition 
(over $\Z$), but make instead an assumption guaranteeing that the usual 
resonance varieties $\RR^m_1(M,\k)$ are defined.

\begin{proposition}
\label{prop:res2-mfd-zero}
Let $M$ be a closed manifold of dimension $m$. Suppose 
$H_1(M,\Z)$ has no $2$-torsion and $\ch(\k)=2$. Then $\RR^m_1(M,\k)= \{0\}$.
\end{proposition}

\begin{proof}
By Poincar\'e duality, the cohomology algebra $H^{\hdot}(M,\Z_2)$ is an 
$m$-$\pda$. Viewing it as an $m$-$\pdcdga$ with $\D=0$,
Theorem \ref{thm:mpd}, Part \eqref{respd3} implies that 
$\RR^m_1(M,\Z_2)= \{0\}$, and so $\RR^m_1(M,\k)= \{0\}$, too. 

The claim can also be proven directly, as follows. Let $\omega=1^{\vee}$ 
be the generator of $H^m(M,\Z_2)=\Z_2$. Let $a\in H^1(M,\Z_2)$ and 
let $a^{\vee}\in H^{m-1}(M,\Z_2)$ be its Poincar\'{e} dual. 
Then $\delta_a(a^{\vee})=aa^{\vee}=\omega$, and so $H^m(A,\delta_a)=0$, 
once again showing that $\RR^m_1(M,\Z_2)= \{0\}$.
\end{proof}

For the next result, we impose orientability (over $\Z$) but no other additional 
conditions on $M$, and draw some conclusions regarding the resonance varieties 
$\RRR^m_1(M,\k)$.

\begin{proposition}
\label{prop:res2-mfd}
Let $M$ be a closed, orientable manifold of dimension $m$, 
and assume $\ch(\k)=2$.  Then $\RRR^q_s(M;\k)=\RRR^{m-q}_s(M;\k)$ 
for all $q,s\ge 0$. In particular, $\RRR^m_1(M,\k)= \{0\}$.
\end{proposition}

\begin{proof}
By Corollary \ref{cor:pd-mfd}, the $\cdga$ $\big(H^{\hdot}(M,\Z_2),\beta_2\big)$ 
is an $m$-$\pdcdga$. Theorem \ref{thm:mpd} then gives 
$\RRR^q_s(M;\Z_2)=\RRR^{m-q}_s(M;\Z_2)$ for all 
$q,s\ge 0$, and the conclusions follow.
\end{proof}

\begin{example}
\label{ex:lens}
The lens space $L(4,1)=S^3/\Z_4$ and the projective space $\RP^3=S^3/\Z_2$ are 
both closed, orientable manifolds of dimension $3$. They share the same $\Z_2$-cohomology 
groups, but the multiplicative structure and the action of the Bockstein are different; indeed, 
$H^{\hdot}(L(4,1),\Z_2)=\Z_2[a,b]/(a^2,b^2)$ with $\abs{a}=1$, $\abs{b}=2$, and 
$\beta_2=0$, whereas $H^{\hdot}(\RP^3,\Z_2)=\Z_2[a]/(a^4)$ with 
$\abs{a}=1$ and $\beta_2(a)=a^2$. It follows that 
$\RRR^1_1(L(4,1),\Z_2)=\RRR^2_1(L(4,1),\Z_2)=\{0\}$, yet, as noted previously,  
$\RRR^1_1(\RP^3,\Z_2)=\RRR^2_1(\RP^3,\Z_2)=\emptyset$.
\end{example}

\begin{example}
\label{ex:dold}
For each $m,n\ge 0$, Dold constructed in \cite{Dold} a smooth, closed 
manifold $P(m,n)$, defined as the quotient of $S^m \times \CP^n$ by the involution 
that acts as the antipodal map on $S^m$ and complex conjugation of $\CP^n$. 
The Dold manifolds, which generate the unoriented cobordism ring, are 
orientable if and only if $m+n$ is odd. The cohomology algebra 
$H^{\hdot}(P(m,n),\Z_2)$ is isomorphic to $\Z_2[a,b]/(a^{m+1},b^{n+1})$, 
where $\abs{a}=1$ and $\abs{b}=2$, and the Bockstein is 
given by $\beta_2(a)=a^2$ and $\beta_2(b)=ab$. It follows that 
$\RRR^q_1(P(m,n),\Z_2) =\Z_2$ if $q<m+2n$ and $q$ is even 
or $q=m+2n$ and $m+n$ is even; 
$\RRR^q_1(P(m,n),\Z_2)=\{0\} $ if $0\le q\le m+2n$;  
and $\RRR^q_s(P(m,n),\Z_2)=\emptyset$, otherwise.
\end{example}

\subsection{A resonant characterization of orientability}
\label{subsec:res-or}
We conclude this section with a result that expresses the 
orientability of a manifold $M$ in terms of its  
top-degree resonance variety $\RRR^m_1(M,\Z_2)$.

\begin{proposition}
\label{prop:res2-nonor-mfd}
A smooth, closed manifold $M$ of dimension $m$ is orientable 
if and only if $\RRR^m_1(M,\Z_2) = \{0\}$.
\end{proposition}

\begin{proof}
If $M$ is orientable, then Proposition \ref{prop:res2-mfd} shows that 
$\RRR^m_1(M,\Z_2)= \{0\}$. On the other hand, if $M$ is not orientable, 
then the Stiefel--Whitney class $w_1=w_1(M)\in H^1(M,\Z_2)$ is non-zero, and 
\begin{gather}
\begin{aligned}
\label{eq:delta-w1}
\delta_{w_1} (u) & = w_1 u + \Sq^1(u)\\
&= v_1 u + \Sq^1(u)\\
&= \Sq^1( u) + \Sq^1(u)\\
&=0
\end{aligned}
\end{gather}
for every $u\in H^{m-1}(M,\Z_2)$, by the Wu formulas \eqref{eq:wu}. 
Let $\omega=1^{\vee}\in H^m(M,\Z_2)$; since $\delta_{w_1} (\omega)=0$, 
this shows that $w_1\in \RRR^m_1(M,\Z_2)$, 
and we are done.
\end{proof}

\begin{corollary}
\label{cor:r-rr}
Let $M$ be a smooth, closed, non-orientable manifold of dimension $m$. Suppose 
$H_1(M,\Z)$ has no $2$-torsion. Then $\RR^m_1(M,\Z_2)= \{0\}$ 
whereas $\RRR^m_1(M,\Z_2)= \Z_2$.
\end{corollary}

\begin{proof}
Proposition \ref{prop:res2-mfd-zero} shows that $\RR^m_1(M,\Z_2)= \{0\}$. 
The proof of Proposition \ref{prop:res2-nonor-mfd} shows that 
$\{0,w_1\}\subseteq\RRR^m_1(M,\Z_2)$. It remain to show 
that this inclusion is an equality.

It follows from the proof of Lemma \ref{lem:bock-or} 
and the discussion from Remark \ref{rem:wu} that the map 
$\Sq^1\colon H^{m-1}(M,\Z_2)\to H^m(M,\Z_2)$ sends $w_1^{\vee}$ to $\omega$ 
and any other element  to $0$. Thus, if $a\in H^1(M,\Z_2)$ is such that  
$a\ne 0$ and $a\ne w_1$, then $a^{\vee}\ne w_1^{\vee}$, and so
\begin{equation}
\label{eq:aavee}
\delta_a(a^{\vee}) = a a^{\vee} + \Sq^1(a^{\vee} ) =\omega+0=\omega.
\end{equation}
Since $\omega$ generates $H^m(M,\Z_2)=\Z_2$, it follows that $a\notin \RRR^m_1(M,\Z_2)$, 
and the proof is complete.
\end{proof}

\begin{example}
\label{ex:non-orient} 
Let $G$ be a finitely presented group which admits an epimorphism 
$\nu\colon G\surj \Z_2$. As explained in \cite{Bud}, standard surgery 
techniques, suitably adapted to this situation, produce a 
smooth, closed, non-orientable, $4$-dimensional manifold $M$ 
with $\pi_1(M)=G$ and with $w_1(M)$ corresponding to $\nu$ under 
the identification $H^1(M,\Z_2)=\Hom(G,\Z_2)$. If, furthermore, 
$G_{\ab}$ has no $2$-torsion (for instance, if $G=\Z^n$, for 
some $n\ge 1$), then $M$ satisfies the hypothesis of the above 
corollary, and so $\RR^4_1(M,\Z_2)= \{0\}$,  
yet $\RRR^4_1(M,\Z_2)= \Z_2$. 
\end{example}

\section{Resonance varieties and Betti numbers of finite covers}
\label{sect:covers}

In this last section we compare the Betti numbers and the resonance varieties 
of certain finite covers to those of the base space.

\subsection{Finite, regular covers and resonance}
\label{subsec:transfer}
Let $X$ be a connected, finite-type CW-complex and let 
$p\colon Y\to X$ be a connected, regular cover, with group 
of deck transformations $\Gamma$.  We assume $\Gamma$ 
is finite, and we let $\k$ be a field of characteristic $0$ or a 
prime not dividing the order of $\Gamma$. Given these data, 
a transfer argument shows that the induced homomorphism 
$p^*\colon H^{\hdot}(X,\k)\to H^{\hdot}(Y,\k)$ is injective, 
with image the subgroup $H^{\hdot}(Y, \k)^{\Gamma}$ of cohomology 
classes invariant under the action of $\Gamma$, 
see e.g.~\cite[Proposition~3G.1]{Ha}.

The next proposition and its corollary were proved in \cite{DP11} 
in the case when $\k$ has characteristic $0$.  For completeness, 
we include a proof, following the argument given in \cite{Su-toulouse}.

\begin{proposition}[\cite{DP11,Su-toulouse}]
\label{prop:jumpcov}
Let $\Gamma$ be a finite group and let $p\colon Y\to X$ 
be a regular $\Gamma$-cover.  Suppose that $\ch(\k)\nmid \abs{\Gamma}$, 
and also $\ch(\k)\ne 2$ if $H_1(X,\Z)$ has $2$-torsion. Then
$p^*(\RR^q_s(X, \k))\subseteq \RR^q_s(Y, \k)$, 
for all $i, s\ge 0$. 
\end{proposition}

\begin{proof}
For a class $a \in H^{1}(Y, \k)^{\Gamma}=H^1(X,\k)$, the monodromy action 
of $\Gamma$  on $H^{\hdot}(Y, \k)$ gives rise to an action on the chain 
complex $(H^{\hdot}(Y, \k), \delta_a)$, with the fixed subcomplex 
equal to $(H^{\hdot}(X, \k), \delta_a)$.  Since $\Gamma$ is finite 
and since $\ch(\k)\nmid \abs{\Gamma}$, a transfer argument 
shows once again that $H^{\hdot}(H^{\hdot}(Y, \k), \delta_a)^{\Gamma}=
H^{\hdot}(H^{\hdot}(X, \k), \delta_a)$. We thus obtain an inclusion,  
$H^{\hdot}(H^{\hdot}(X, \k), \delta_a) \hookrightarrow 
H^{\hdot}(H^{\hdot}(Y, \k), \delta_a)$, and the claim follows. 
\end{proof}

\begin{corollary}
\label{cor:trivnom}
With notation as above, suppose the group $\Gamma$ acts trivially on $H^1(Y,\k)$.  
Then the map $p^*\colon \RR^1_s(X, \k)\to \RR^1_s(Y, \k)$ is an isomorphism, 
for all $s\ge 0$. 
\end{corollary}

By way of contrast, the resonance varieties $\RRR^q_s(Y,\Z_2)$ may vanish, even 
when the corresponding varieties $\RRR^q_s(X,\Z_2)$ are non-zero. This 
phenomenon is illustrated by the next result, which is an immediate corollary 
to Proposition \ref{prop:res2-nonor-mfd}.

\begin{corollary}
\label{cor:orient-cover}
Let $M$ be a smooth, closed, non-orientable manifold of dimension $m$, and let 
$\widetilde{M}$ be its orientation double cover. Then $\RRR^m_1(M,\Z_2)\ne \{0\}$ 
yet $\RRR^m_1(\widetilde{M},\Z_2)=\{0\}$.
\end{corollary}

\subsection{Mod-$2$ Betti numbers of $2$-fold covers}
\label{subsec:2-fold}

We now consider the case when the order of the deck group $\Gamma$ 
is not coprime to the characteristic of the coefficient field $\k$. 
We will focus on the simplest possible situation: the one 
where $\Gamma$ is a finite cyclic group of even order 
(mainly, $\Gamma=\Z_2$) and $\k=\Z_2$.

Every regular $\Z_n$-cover $p\colon Y\to X$ is classified by a 
homomorphism $\alpha\colon \pi_1(X)\to \Z_n$, or, equivalently, 
by a cohomology class (called its characteristic class), $\alpha\in H^1(X,\Z_n)$. 
Furthermore, the covering space $Y$ is connected if and only if the 
homomorphism $\alpha$ is surjective, in which case $\pi_1(Y)=\ker(\alpha)$. 
Note also that all $2$-fold covers are regular.

We start with a simple lemma, which strengthens Lemma 3.2, part (i) 
from \cite{Yo20}.

\begin{lemma}
\label{lem:2-fold-cover}
Let $p\colon Y\to X$ be a connected $\Z_2$-cover, with characteristic class 
$\alpha\in H^1(X,\Z_2)$. Then $p$ lifts to a connected, regular $\Z_4$-cover 
$\bar{p}\colon \overline{Y}\to X$ if and only if $\alpha^2=0$.
\end{lemma}

\begin{proof}
Since $Y$ is connected, the homomorphism 
$\alpha\colon \pi_1(X)\to \Z_2$ is surjective, or, equivalently, 
the characteristic class $\alpha$ is non-zero. As noted in \S\ref{subsec:bock}, 
the Bockstein of $\alpha$ is given by $\beta_2(\alpha)=\alpha^2$. Thus, $\alpha^2=0$ 
if and only if there is a class $\bar\alpha\in H^1(X,\Z_4)$ such that $\alpha$ is the 
reduction mod-$2$ of $\bar\alpha$. Let $\bar{p}\colon \overline{Y}\to X$ be the 
corresponding regular $\Z_4$-cover. Since $\alpha\ne 0$, the homomorphism 
$\bar\alpha\colon \pi_1(X)\to \Z_4$ must be surjective, and so $\overline{Y}$ 
must be connected. 
\end{proof} 

The next result generalizes Lemma 3.2, part (ii) and Theorem 3.7 from \cite{Yo20}.

\begin{proposition}
\label{prop:betti-cover}
Let $p\colon Y\to X$ be a $2$-fold cover, classified by a non-zero class 
$\alpha\in H^1(X,\Z_2)$. Suppose that $\alpha^2=0$. 
Then, for all $q\ge 1$, 
\begin{equation}
\label{eq:hcov}
b_q(Y,\Z_2)= b_q(X,\Z_2) + \dim_{\Z_2} H^q(H^{\hdot}(X,\Z_2),\delta_{\alpha}).
\end{equation}
In particular, $b_q(Y,\Z_2)\ge b_q(X,\Z_2)$. 
\end{proposition}

\begin{proof}
The transfer exact sequence for the $2$-fold cover with characteristic class $\alpha$ 
(see \cite{Ha, Hs}) takes the form
\begin{equation}
\label{eq:tr-seq}
\begin{tikzcd}[column sep=13.5pt]
\!\cdots\!  \ar[r]&H^q(X,\Z_2) \ar[r, "\cdot \alpha\,"]& H^q(X,\Z_2) \ar[r, "p^*"]& H^q(Y,\Z_2) \ar[r, "\tau"]& 
H^q(X,\Z_2) \ar[r, "\cdot \alpha\,"]& H^{q+1}(X,\Z_2) \ar[r]& \!\cdots\! .
\end{tikzcd}
\end{equation}
Since by assumption $\beta_2(\alpha)=\alpha^2$ vanishes, we have that 
$\delta_\alpha(u)=\alpha u$ for all $u\in H^i(X,\Z_2)$. Hence, 
the above long exact sequence splits into short exact sequences, 
\begin{equation}
\label{eq:short-seq}
\begin{tikzcd}[column sep=16pt]
0  \ar[r]& \im(\delta_{\alpha}) \ar[r]& H^q(X,\Z_2) \ar[r, "p^*"]& H^q(Y,\Z_2) \ar[r]& 
\ker(\delta_{\alpha})\ar[r]& 0 ,
\end{tikzcd}
\end{equation}
and the claim follows.
\end{proof}

\begin{remark}
\label{rem:yoshi}
Yoshinaga \cite{Yo20} proved this result in the case when either $q=1$ or 
$X$ is the complement of a complex hyperplane arrangement. He used it to detect 
a $\Z_2$-summand in the first homology of the Milnor fiber of the icosidodecahedral
arrangement.
\end{remark}

If the $2$-fold cover $Y\to X$ is classified by an element 
$\alpha\in H^1(X,\Z_2)$ with $\alpha^2\ne 0$, the mod-$2$ Betti 
numbers of $Y$ may actually be smaller than those of $X$.

\begin{example}
\label{ex:sn-rpn}
For $n>1$, let $S^n\to \RP^n$ be the orientation double cover, classified by 
the degree-$1$ generator $\alpha=w_1(\RP^n)$ of the cohomology algebra 
$H^{\hdot}(\RP^n, \Z_2)=\Z_2[\alpha]/(\alpha^{n+1})$. 
\end{example}

\section*{Note added in proof}
\label{note:milnor-fiber}
After acceptance of this paper, Ishibashi, Sugawara, and Yoshinaga posted 
a preprint, \cite{ISY}, in which they expand on the work described in 
Section \ref{subsec:2-fold}. Namely, suppose $Y\to X$ is a connected $\Z_2$-cover 
with characteristic class $\alpha\in H^1(X,\Z_2)$. Assuming $H_{\hdot}(X,\Z)$ is 
torsion-free, it follows from \cite[Theorem C]{PS-tams} that 
\begin{equation}
\label{eq:modular}
b_q(Y) \le  b_q(X) + \dim_{\Z_2} H^q(H^{\hdot}(X,\Z_2),\delta_{\alpha}).
\end{equation}

When $X$ is the (projectivized) complement of a hyperplane arrangement, 
an explicit formula was proposed in \cite[Conjecture 1.9]{PS-plms17}, expressing 
the first Betti number of the Milnor fiber of the arrangement in terms of 
the resonance varieties $\RR^1_s(X,\Z_p)$, for $p=2$ and $3$. At the 
prime $p=2$, the conjecture is equivalent to the inequality \eqref{eq:modular} 
holding as equality in degree $q=1$ for the $2$-fold cover $Y\to X$ corresponding 
to the class $\alpha\in H^1(X,\Z_2)$ which evaluates to $1$ on each meridian. 
As shown in \cite[Corollary 2.5]{ISY}, though, that equality holds if and only if $H_1(Y,\Z_2)$ 
has no $2$-torsion. Therefore, the formula conjectured in \cite{PS-plms17} fails  
at the prime $p=2$ precisely when $H_1(Y,\Z_2)$ has non-trivial $2$-torsion, 
as is the case in the example mentioned in Remark \ref{rem:yoshi}.

\newcommand{\arxiv}[1]
{\texttt{\href{http://arxiv.org/abs/#1}{arXiv:#1}}}
\newcommand{\arxi}[1]
{\texttt{\href{http://arxiv.org/abs/#1}{arxiv:}}
\texttt{\href{http://arxiv.org/abs/#1}{#1}}}
\newcommand{\doi}[1]
{\texttt{\href{http://dx.doi.org/#1}{doi:#1}}}
\renewcommand{\MR}[1]
{\href{http://www.ams.org/mathscinet-getitem?mr=#1}{MR#1}}
\newcommand{\MRh}[2]
{\href{http://www.ams.org/mathscinet-getitem?mr=#1}{MR#1 (#2)}}


\begin{thebibliography}{00}

\bibitem{BD} M.~Brickenstein, A.~Dreyer,
\href{https://doi.org/10.1016/j.jsc.2011.04.002}%
{\em Gr\"{o}bner-free normal forms for Boolean polynomials}, 
J. Symbolic Comput. \textbf{48} (2013), 37--53. 
\MR{2980465}

\bibitem{B-W} M.~Brickenstein, A.~Dreyer, G.-M.~Greuel, 
M.~Wedler, O.~Wienand, 
\href{https://dx.doi.org/10.1016/j.jpaa.2008.11.043}
{\em New developments in the theory of Gr\"{o}bner bases and 
applications to formal verification}, 
J. Pure Appl. Algebra \textbf{213} (2009), no.~8, 1612--1635. 
\MR{2517997} 

\bibitem{Bud} R.~Budney, 
{\em Fundamental groups of non-orientable closed four-manifolds}, 
MathOverflow (October 19, 2017), \url{https://mathoverflow.net/q/283825}.

\bibitem{CS99} D.C.~Cohen, A.I.~Suciu, 
\href{https://doi.org/10.1017/S0305004199003576}%
{\em Characteristic varieties of arrangements},
Math. Proc. Cambridge Phil. Soc. \textbf{127} (1999), 
no.~1, 33--53. 
\MR{1692519}

\bibitem{DS-models} G.~Denham, A.I.~Suciu,
{\em Algebraic models and cohomology jump loci}, 
in preparation.

\bibitem{DP11} A.~Dimca, S.~Papadima, 
\href{https://www.numdam.org/item/ASNSP_2011_5_10_2_253_0/}%
{\em Finite {G}alois covers, cohomology jump loci, formality 
properties, and multinets},  Ann. Sc. Norm. Super. Pisa Cl. Sci. 
\textbf{10} (2011), no.~2, 253--268. 
\MR{2856148}

\bibitem{Dold} A.~Dold, 
\href{https://dx.doi.org/10.1007/BF01473868}%
{\em Erzeugende der Thomschen Algebra $\mathcal {N}$}, 
Math. Z. \textbf{65}  (1956), no.~1, 25--35. 
\MR{0079269}

\bibitem{Fa97} M.~Falk,
\href{https://dx.doi.org/10.1007/BF02558471}%
{\em Arrangements and cohomology},
Ann. Combin. \textbf{1} (1997), no.~2, 135--157.  
\MR{1629681}

\bibitem{Fa07} M.~Falk, 
\href{https://dx.doi.org/10.1093/imrn/rnm009}
{\em Resonance varieties over fields of positive characteristic},
Int. Math. Res. Not. IMRN \textbf{2007}, no.~3, 
article ID rnm009, 25 pages.
\MR{2337033} 

\bibitem{Gr} R.~Greenblatt, 
\href{https://projecteuclid.org/euclid.hha/1175791075}
{\em Homology with local coefficients and characteristic 
classes}, Homology Homotopy Appl. \textbf{8} (2006), 
no.~2, 91--103. 
\MR{2246023} 

\bibitem{Ha} A.~Hatcher, 
\href{https://www.cambridge.org/us/academic/subjects/mathematics/geometry-and-topology/algebraic-topology-1?format=PB&isbn=9780521795401}%
{\em Algebraic topology}, Cambridge University Press, 
Cambridge, 2002.
\MR{1867354} 

\bibitem{Hs} J.-Cl.~Hausmann, 
\href{https://doi.org/10.1007/978-3-319-09354-3}%
{\em Mod two homology and cohomology}, 
Universitext, Springer, Cham, 2014. 
\MR{3308717}

 \bibitem{ISY} S.~Ishibashi, S.~Sugawara, M.~Yoshinaga, 
 {\em Betti numbers and torsions in homology groups of double coverings}, 
 \arxiv{2209.02236}.

\bibitem{LS-aif}  P.~Lambrechts, D.~Stanley,
\href{https://doi.org/10.5802/aif.2042}%
{\em The rational homotopy type of configuration spaces of two points}, 
Ann. Inst. Fourier (Grenoble) \textbf{54} (2004), no.~4, 1029--1052. 
\MR{2111020}

\bibitem{LS-asens}  P.~Lambrechts, D.~Stanley,
\href{https://doi.org/10.24033/asens.2074}%
{\em Poincar\'{e} duality and commutative differential graded algebras},
Ann. Scient. \'{E}c. Norm. Sup. \textbf{41} (2008), no.~4, 497--511.
\MR{2489632}

\bibitem{LY00}  A.~Libgober, S.~Yuzvinsky, 
\href{https://doi.org/10.1023/A:1001826010964}%
{\em Cohomology of the {O}rlik-{S}olomon algebras and local systems}, 
Compositio Math. \textbf{21} (2000), no.~3, 337--361. 
\MR{1761630}

\bibitem{Lu} S.~Lundqvist,
\href{https://dx.doi.org/10.1016/j.jpaa.2015.02.030}%
{\em Boolean ideals and their varieties}, 
J. Pure Appl. Algebra \textbf{219} (2015), no.~10, 4521--4540.
\MR{3346504}

\bibitem{MPPS}  A.~M\u{a}cinic, S.~Papadima, R.~Popescu,  A.I.~Suciu, 
 \href{https://dx.doi.org/10.1090/tran/6799}%
{\em Flat connections and resonance varieties: 
from rank one to higher ranks},  Trans. Amer. Math. Soc. 
\textbf{369} (2017), no.~2, 1309--1343.    
\MR{3572275}

\bibitem{Mas} W.S.~Massey,
\href{https://doi.org/10.2307/2372878}%
{\em On the Stiefel--Whitney classes of a manifold}, 
Amer. J. Math. \textbf{82} (1960), no.~1, 92--102.
\MR{0111053}

\bibitem{MS00}  D.~Matei, A.I.~Suciu,
\href{https://doi.org/10.2969/aspm/02710185}%
{\em Cohomology rings and nilpotent quotients of real and
complex arrangements}, in: Arrange\-ments--Tokyo 1998, 
pp.~185--215, Adv. Stud. Pure Math., vol.~27, 
Math. Soc. Japan, Tokyo, 2000.
\MR{1796900} 

\bibitem{MeS} D.~Meyer, L.~Smith, 
\href{https://www.cambridge.org/us/academic/subjects/mathematics/algebra/poincare-duality-algebras-macaulays-dual-systems-and-steenrod-operations}%
{\em Poincar\'{e} duality algebras, Macaulay's dual systems, 
and Steenrod operations}, Cambridge Tracts in Math., 
vol.~167, Cambridge University Press, Cambridge, 2005. 
\MR{2177162} 

\bibitem{Milnor} J.W.~Milnor, J.D.~Stasheff, 
\href{https://press.princeton.edu/books/paperback/9780691081229/characteristic-classes-am-76-volume-76}%
{\em Characteristic classes}, Annals of Mathematics Studies, vol.~76, Princeton University Press, 
Princeton, NJ, 1974. 
\MR{0440554}

\bibitem{PS-tams} S.~Papadima, A.I.~Suciu,
\href{https://dx.doi.org/10.1090/S0002-9947-09-05041-7}%
{\em The spectral sequence of an equivariant chain 
complex and homology with local coefficients}, 
Trans.~Amer.~Math.~Soc. \textbf{362} (2010), 
no.~5, 2685--2721.
\MR{2584616}

\bibitem{PS-plms10} S.~Papadima, A.I.~Suciu,
\href{https://dx.doi.org/10.1112/plms/pdp045}%
{\em Bieri--{N}eumann--{S}trebel--{R}enz invariants and 
homology jumping loci}, Proc.~London Math.~Soc. 
\textbf{100} (2010), no.~3, 795--834.
\MR{2640291}  

\bibitem{PS-springer} S.~Papadima, A.I.~Suciu,
\href{https://dx.doi.org/10.1007/978-3-319-09186-0_17}%
{\em Non-abelian resonance: product and coproduct formulas}, 
in: Bridging Algebra, Geometry, and Topology, pp.~269--280, 
Springer Proc. Math. Stat., vol. 96, 2014. 
\MR{3297121}

\bibitem{PS-plms17} S.~Papadima, A.I.~Suciu,
\href{http://dx.doi.org/10.1112/plms.12027}%
{\em The Milnor fibration of a hyperplane arrangement: from 
modular resonance to algebraic monodromy}, 
Proc.~London Math.~Soc.  \textbf{114} (2017), no.~6, 961--1004.
\MR{3661343}

\bibitem{PS-imrn} S.~Papadima, A.I.~Suciu, 
\href{https://dx.doi.org/10.1093/imrn/rnx294}%
{\em The topology of compact Lie group actions through the lens 
of finite models}, Int. Math. Res. Notices IMRN 
\textbf{2019} (2019), no.~20, 6390--6436.
\MR{4031243}

\bibitem{Po} M.M.~Postnikov, 
{\em The structure of the ring of intersections of three-dimensional manifolds}, 
Doklady Akad. Nauk. SSSR \textbf{61} (1948), 795--797. 
\MR{0027513}

\bibitem{Sam} H.~Samelson, 
\href{https://dx.doi.org/10.2307/2034522}%
{\em A note on the Bockstein operator}, 
Proc. Amer. Math. Soc. \textbf{15} (1964), 450--453.
\MR{0164345}

\bibitem{SS10} L.~Smith, R.E.~Stong, 
\href{https://dx.doi.org/10.1016/j.aim.2010.04.013}%
{\em Poincar\'{e} duality algebras mod two}, 
Adv. Math. \textbf{225} (2010), no.~4, 1929--1985. 
\MR{2680196}

\bibitem{Su-conm} A.I.~Suciu,
\href{https://dx.doi.org/10.1090/conm/276/04510}%
{\emph{Fundamental groups of line arrangements: enumerative aspects}},
in: {\em Advances in algebraic geometry motivated by physics} ({L}owell, {MA}, 2000),
43--79, Contemp. Math., vol. 276, Amer. Math. Soc., Providence, RI, 2001.
\MR{1837109}

\bibitem{Su12} A.I.~Suciu,
\href{https://dx.doi.org/10.2969/aspm/06210359}%
{\em Resonance varieties and Dwyer--Fried invariants}, 
in: {\em Arrangements of Hyperplanes (Sapporo 2009)},  
359--398, Advanced Studies Pure Math., vol.~62, Kinokuniya, 
Tokyo, 2012. 
\MR{2933803}

\bibitem{Su-toulouse} A.I.~Suciu,
\href{https://dx.doi.org/10.5802/afst.1412}%
{\em Hyperplane arrangements and {M}ilnor fibrations}, 
Ann. Fac. Sci. Toulouse Math. (6) \textbf{23} (2014), no.~2, 417--481.  
\MR{3205599}

\bibitem{Su16} A.I.~Suciu, 
\href{https://dx.doi.org/10.1007/978-3-319-31580-5_1}%
{\em Around the tangent cone theorem}, in: 
Configuration Spaces: Geometry, Topology and Representation Theory, 
1--39, Springer INdAM series, vol.~14, Springer, Cham, 2016.
\MR{3615726}

\bibitem{Su-edinb} A.I.~Suciu, 
\href{https://doi.org/10.1017/prm.2019.55}%
{\em Poincar\'{e} duality and resonance varieties}, 
Proc. Roy. Soc. Edinburgh Sect. A. 
\textbf{150} (2020), nr.~6, 3001--3027. 
\MR{4190099}

\bibitem{Su-sullivan} A.I.~Suciu,
{\em Formality and finiteness in rational homotopy theory}, 
\arxiv{2210.08310}, to appear in EMS Surveys in Mathematical Sciences.

\bibitem{Su-cdga} A.I.~Suciu,
{\em Alexander invariants and holonomy Lie algebras of 
commutative differential graded algebras}, in preparation.

\bibitem{SW-mz} A.I.~Suciu, H.~Wang,
\href{https://dx.doi.org/10.1007/s00209-016-1811-x}%
{\em Pure virtual braids, resonance, and formality}, 
Math. Z. \textbf{286} (2017), no.~3--4, 1495--1524.
\MR{3671586}

\bibitem{Yo20}  M.~Yoshinaga, 
\href{https://dx.doi.org/10.1007/s40879-019-00387-8}%
{\em Double coverings of arrangement complements and $2$-torsion 
in Milnor fiber homology},  Eur. J. Math. \textbf{6} (2020), nr.~3, 
1097--1109.
\MR{4151730}

\end{thebibliography}
\end{document}